\documentclass[12pt,reqno]{amsart}

\numberwithin{equation}{section}

\usepackage[T1]{fontenc}
\usepackage{ae,aecompl}
\usepackage{graphics,cancel,graphicx}
\usepackage{epsfig,psfrag,manfnt}
\usepackage{mathrsfs}
\usepackage{graphicx,mathabx}
\usepackage{bbm,subfigure,rotating,bm,float,mathdots,wasysym,scalerel}

\makeatletter
\@namedef{subjclassname@2020}{\textup{2020} Mathematics Subject Classification}
\makeatother

\usepackage{tikz-cd}
\usepackage{tikz}
\usetikzlibrary{positioning} 

\usepackage{stmaryrd}
\usepackage[all,cmtip]{xy}
\usepackage{amssymb,amsmath,amsfonts,amsthm,color}

\usepackage{scalerel,stackengine}
\stackMath
\newcommand\reallywidehat[1]{%
\savestack{\tmpbox}{\stretchto{%
  \scaleto{%
    \scalerel*[\widthof{\ensuremath{#1}}]{\kern-.6pt\bigwedge\kern-.6pt}%
    {\rule[-\textheight/2]{1ex}{\textheight}}
  }{\textheight}%
}{0.5ex}}%
\stackon[1pt]{#1}{\tmpbox}%
}
\newlength{\wdth}

\newcommand{\mC}{\mathbb{C}}
\newcommand{\mD}{\mathbb{D}}

\newcommand{\mF}{\mathbb{F}}

\newcommand{\mN}{\mathbb{N}}

\newcommand{\mP}{\mathbb{P}}
\newcommand{\mQ}{\mathbb{Q}}
\newcommand{\mR}{\mathbb{R}}

\newcommand{\mT}{\mathbb{T}}

\newcommand{\mZ}{\mathbb{Z}}

\newcommand{\bba}{\mathbf{a}}
\newcommand{\bbb}{\mathbf{b}}

\newcommand{\bbe}{\mathbf{e}}

\newcommand{\bbw}{\mathbf{w}}

\newcommand{\bbz}{\mathbf{z}}

\newcommand{\homeq}{{\begin{smallmatrix}\phantom{g} \\ 
{\scaleobj{1.33}{\sim}} \\[-0.09cm] {\scaleobj{0.72}{h}} \end{smallmatrix}}}

\newtheorem{theorem}{Theorem}[section]
\newtheorem*{Theorem}{Theorem}
\newtheorem*{Corollary}{Corollary}
\newtheorem*{question}{Question}
\newtheorem{lemma}[theorem]{Lemma}
\newtheorem{corollary}[theorem]{Corollary}
\newtheorem{proposition}[theorem]{Proposition}
\newtheorem{conjecture}[theorem]{Conjecture}

\theoremstyle{definition}

\newtheorem{remark}[theorem]{Remark}

\theoremstyle{definition}

\theoremstyle{definition}

\theoremstyle{definition}
\newtheorem{example}[theorem]{Example}

\begin{document}

\keywords{Projective free rings, Hermite rings, algebras of continuous functions, Stein space, 
Banach algebras, maximal ideal space, homotopy equivalence, \v{C}ech cohomology, trivial shape, topological tensor products}

\subjclass[2020]{Primary 46J10; Secondary 46M10, 13C10, 46J15.}

 \title[]{Projective freeness and Hermiteness\\ of complex function algebras}

 \author[]{Alexander Brudnyi}
 \address{Department of Mathematics and Statistics\\
University of Calgary\\
Calgary, Alberta, Canada T2N 1N4}
\email{abrudnyi@ucalgary.ca}
 
 \author[]{Amol Sasane}
\address{Department of Mathematics \\London School of Economics\\
     Houghton Street\\ London WC2A 2AE\\ United Kingdom}
\email{A.J.Sasane@lse.ac.uk}
 
\maketitle
  
 \vspace{-0.5cm}
  
\begin{abstract} 
The paper studies projective freeness and Hermiteness
of algebras of complex-valued continuous functions on topological spaces, Stein algebras, and commutative unital Banach algebras.
New sufficient cohomology conditions on the maximal ideal spaces of the  algebras are given that guarantee the fulfilment of these properties. The results are illustrated by nontrivial examples.  Based on the Borsuk theory of shapes, a new class $\mathscr{C}$ of commutative unital complex Banach algebras is introduced (an analog of the class of local rings in commutative algebra) such that  
the projective tensor product  with algebras in $\mathscr C$ preserves 
projective freeness and Hermiteness.
Some examples of algebras of class $\mathscr{C}$ and of other projective free and Hermite function algebras are assembled. These include, e.g., Douglas algebras, finitely generated algebras of symmetric functions, Bohr-Wiener algebras, algebras of holomorphic semi-almost periodic functions, and algebras of bounded holomorphic functions on Riemann surfaces.
\end{abstract}

\section{Introduction}
In this article, we study projective freeness and Hermiteness of certain function algebras.
\noindent Recall that a unital commutative ring $R$  
is said to be {\em projective free} if every finitely generated projective $R$-module is free 
(i.e., if $M$ is an  $R$-module such that $M\oplus N\cong R^n$  for an $R$-module $N$ and  $n \in \mZ_+\, (:=\mN\cup\{0\})$, then $M\cong R^m$ for some  $m \in \mZ_+$). 
 In terms of matrices, the ring $R$ is projective free if and only if 
 for each $n\in \mathbb{N}$ every $n\times n$ matrix $X\not\in \{0_n, I_n\}$ over $R$ such that $X^2=X$ (i.e., an idempotent)
has a form $X=S(I_r\oplus  0_{n-r})S^{-1}$  for some $ 
S\in GL_n(R)$, $r\in \{1,\dots, n-1\}$; here $GL_n(R)$ denotes the group of invertible $n\times n$ matrices over $R$ and $0_k$ and $I_k$ are zero  and  identity $k\times k$ matrices; see \cite[Proposition~2.6]{Coh0}.

Quillen and Suslin (see, e.g., \cite{Coh}) proved, independently, 
that the polynomial ring over a projective free ring is again projective free; 
in particular settling affirmatively Serre's problem from 1955, 
which asked if  the polynomial ring $\mF[x_1, \dots, x_n]$ is projective free 
for any field $\mF$ (see \cite{Lam}).  Also, if $R$ is any projective
free ring, then the formal power series ring $R\llbracket x\rrbracket$ 
 is again projective free \cite[Theorem~7]{Coh}, and so  
 the formal power series ring $\mF\llbracket x_1, \dots  , x_n\rrbracket$ is projective free.
In the context of rings arising in analysis,  it is known that for example    
the algebra of complex continuous functions on a contractible topological space is projective free \cite{Vas}. According to the Grauert theorem \cite{Gra58} and  the Novodvorski-Taylor theory (see, e.g.,\cite{Nov},   \cite[\S 7.5]{Tay}),  the same holds for the algebra of holomorphic functions on a contractible reduced Stein space and  a commutative unital complex Banach algebra with a contractible maximal ideal space. 
 
 In control theory, projective freeness of rings of stable transfer functions plays an important role in the stabilisation problem, since if the underlying ring of stable transfer functions is projective free, then the stabilisability of an unstable system is equivalent to the existence of a doubly coprime factorisation, see \cite[Theorem~6.3]{Qua}.

The concept of a Hermite ring is a weaker notion than that of a projective free ring. A unital commutative ring $R$ is {\em Hermite} if every finitely generated stably free $R$-module is free. 
An $R$-module $M$ is {\em stably free} if it is stably isomorphic to a free module, i.e., 
$M\oplus R^k $ is free for some $k\in \mZ_+$. 
Since every stably free module is projective, every projective free ring is Hermite.  In terms of matrices, $R$ is Hermite if and only if every left-invertible $n\times k$ matrix, $k,n\in\mN$, $k<n$,  over $R$ can be completed to an invertible one (see, e.g., \cite[p.VIII]{Lam}, \cite[p.1029]{Tol}). 

In control
theory, Hermiteness of the underlying ring of stable transfer functions  implies that if 
the transfer function of an unstable system has a right (or left) coprime factorisation, 
then it has a doubly coprime factorisation; see \cite[Theorem~66]{Vid}.

In the article, we study successively projective freeness and Hermiteness of algebras of complex-valued continuous functions on topological spaces (\S2), Stein algebras (\S3) and commutative unital complex Banach algebras (\S4). We begin with 
a survey of some results  describing the structure of finitely generated projective modules over these algebras in terms of complex vector bundles over their maximal ideal spaces using the Swan and Vaserstein theorems \cite{Sw} and \cite{Vas}, the Grauert Oka principle \cite{Gra57}, \cite{Gra58} and the Novodvorski-Taylor theory \cite{Nov},   \cite{Tay}.  Then, based on this description, we give new sufficient cohomology conditions for 
projective freeness and Hermiteness of the considered algebras and illustrate our results by nontrivial examples.
In \S5, we introduce and study a new class $\mathscr C$  of commutative unital complex Banach algebras $A$ (namely those whose maximal ideal spaces are of trivial shape) such that for every commutative unital complex Banach algebra  $B$ the isomorphism classes of finitely generated projective modules over $B$ are in a one-to-one correspondence with those over the projective tensor product  $A\widehat\otimes_\pi B$. In particular, the projective tensor product with algebras in $\mathscr C$ preserves projective freeness and Hermiteness. Finally, in \S6 several examples of algebras of class $\mathscr{C}$ and of projective free function algebras  are assembled. These include
finitely generated algebras of symmetric functions (\S\ref{subsec_6.1}), Bohr-Wiener algebras (\S\ref{subsec_6.2}), algebras of holomorphic semi-almost periodic functions (\S\ref{subsec_6.3}), and algebras of bounded holomorphic functions on Riemann surfaces (\S\ref{subsec_6.4}).  The Appendix contains
auxiliary results used in \S6.

\section{Algebras of continuous functions}
\label{section_C(X)}

For a topological space $X$, we denote by $C(X)$ the algebra (with pointwise operations) of complex-valued continuous functions on $X$. Also, by $C_b(X)\subset C(X)$ we denote the subalgebra of bounded functions. $C_b(X)$ equipped with
the supremum norm, $\|f\|_\infty=\sup_{x\in X} |f(x)|$, is a complex commutative unital Banach algebra.

We recall that a  bundle  over $X$ is of {\em finite type} if there is a finite set $S$ of nonnegative continuous functions on X whose
sum is $1$ such that the restriction of the bundle to the set $\{x\in X\, :\, f(x)\ne 0\}$ is trivial for each $f$ in $S$.

Every bundle over a compact Hausdorff space $X$ is of finite type. In turn, a bundle over a normal space $X$  is of finite type if and only if there is a finite open covering $\mathfrak U$ of $X$ such that the restriction of the bundle to each $U\in \mathfrak U$ is trivial.

Let ${\textrm{\bf Gr}} ( \mC^{n} )$ be the set of all subspaces of $\mC^n$ equipped with the topology of the disjoint union of the Grassmanian manifolds ${\textrm{\bf Gr}_m} ( \mC^{n} )$, $0\le m\le n$. For each $n\in \mZ_+$ there is a natural embedding ${\textrm{\bf Gr}} ( \mC^{n} )\subset {\textrm{\bf Gr}} ( \mC^{n+1} )$  so that the space ${\textrm{\bf Gr}} ( \mC^{\infty} )=\bigcup_{n\in\mZ_+}\!{\textrm{\bf Gr}} ( \mC^{n} )$, equipped with the inductive limit topology, is well-defined.

The next result follows from Vaserstein's extension \cite[Theorem~2]{Vas} of the Swan theorem \cite{Sw}.

\begin{proposition}\label{prop2.1}
The following statements are equivalent:
\begin{enumerate}
\item The algebra $C(X)$ is projective free.
\item The algebra $C_b(X)$ is projective free.
\item The space $X$ is connected and each finite type complex vector bundle over $X$ is trivial.
\item Each continuous map $X\rightarrow {\textrm{\bf Gr}} ( \mC^{\infty} )$ with a relatively compact image is homotopic, via a homotopy with relatively compact image,  to a constant map. 
\end{enumerate}
\end{proposition}

In what follows, for a topological space $Y$ and an abelian group $G$, 
the $n^{\textrm{th}}$ \v{C}ech cohomology group of $Y$ with values in $G$ is denoted by $H^n(Y,G)$, $n\in \mZ_+$.

Let $\beta X$ be the Stone-\v{C}ech compactification of $X$. Then $\beta X$ is naturally homeomorphic to the maximal ideal space of $C_b(X)$, and $C(\beta X)$ is isometrically isomorphic to $C_b(X)$. Proposition \ref{prop2.1} implies that if $C(X)$ is projective free, then $C(\beta X)$ is projective free, and so $\beta X$ is connected and all finite rank complex vector bundles over $\beta X$ are trivial.

Since isomorphism classes of rank one complex vector bundles over $\beta X$ are in one-to-one correspondence (determined by assigning to a bundle its first Chern class) with elements of  
$H^2(\beta X, \mZ)$,  projective freeness of $C(X)$ implies that $H^2(\beta X, \mZ) = 0$.  Also, in this case  Proposition~\ref{prop2.1} implies that the Grothendieck group $K_0(C_b(X))$ of the algebraic $K$-theory for $C_b(X)$ (isomorphic to the group $K^0(\beta X)$ of the Atiyah-Hirzebruch theory)  
is $\mathbb Z$. Since the Chern character $\textrm{ch}$ determines a ring isomorphism $K_0(C_b(X))\otimes \mQ\rightarrow H^{\textrm{even}}(\beta X,\mQ)$, 
see, e.g., \cite[Theorem~3, p.16-09]{Kar}, we obtain the following result. 

\begin{corollary}
\label{16_may_2022_14:29}
 If the algebra $C(X)$ is projective free$,$ then $\beta X$ is connected$,$  $H^2(\beta X, \mZ) = 0,$ and $H^{2n}(\beta X,\mQ)=0$ for all $n\ge 2$.
 \end{corollary}
 
 For instance, if $X$ is a connected closed orientable manifold such that the algebra $C(X)$ is projective free$,$ then the corollary along with the Poincar\'{e} duality imply that $X$ has an  odd dimension $n$  and $H^{k}( X,\mQ)=0$ for all $k\not\in\{0,n\}$. Note that even if  $H^{k}(X,\mZ)=0$ for all $k\not\in\{0,n\}$, $C(X)$ is not necessarily projective free.  In fact, it is not even necessarily Hermite; see Example~\ref{14_May_2022_14:40} below.
 
 Let $\Omega^n=X\times\mC^n$ denote the standard trivial rank $n$ complex vector bundle  over a  topological space $X$. A complex vector bundle $E$ over $X$ is {\em stably trivial} if there exist $m,n\in \mZ_+$ such that the bundle $E\oplus \Omega^m$ is isomorphic to $\Omega^n$. 
 
Our next result can be easily deduced from \cite[Theorem~2]{Vas} as well.

 \begin{proposition}\label{prop2.3}
 The following statements are equivalent:
\begin{enumerate}
\item The algebra $C(X)$ is Hermite.
\item The algebra $C_b(X)$ is Hermite.
\item Each stably trivial finite type complex vector bundle over $X$ is trivial.
\end{enumerate}
\end{proposition}

Propositions \ref{prop2.1} and \ref{prop2.3} imply the following result.

\begin{corollary}\label{cor2.4}
 Projective freeness and Hermiteness of the algebra $C(X)$ depend only on the homotopy type of $X$.
\end{corollary}

\begin{example}
\label{14_May_2022_14:40}
Let $S^n$ denote the $n$-dimensional unit sphere.  
Cutting the sphere $S^n$  into two `bowls' (each homeomorphic to an $n$-dimensional ball) 
 such that their intersection is a `collar' $C$ containing the subsphere $S^{n-1}$, 
 it can be seen that 
the isomorphism classes $\textrm{Vect}_k(S^n)$ of rank $k$
complex vector bundles on $S^n$ are in a one-to-one correspondence with 
the elements of  the  $(n-1)^{\textrm{st}}$ 
homotopy group $\pi_{n-1}(U(k))$ of the unitary group $U(k)$. 
As the homotopy group $\pi_4(U(2))\cong \pi_3(S^4)\cong \mZ_2$, it follows that 
there exists a nontrivial rank $2$ complex vector bundle $E$ over $S^5$.
Hence, $C(S^5)$ is not projective free, although $H^{k}(S^5,\mZ)=0$ for all $k\not\in \{0,5\}$. Moreover, since $K_0(C(S^5))=K^0(S^5)=\mZ$ 
(see, e.g., \cite[Corollary~5.2 on p.121 and p.150 of Chap. 11]{Hus}), every finite rank complex vector bundle on $S^5$ is stably trivial. Hence, $C(S^5)$ is not Hermite, as $E$ is nontrivial. 
\end{example}

Next, we formulate an auxiliary result used in our proofs which follows easily, e.g., from 
 \cite[Theorem~11.9, p.287]{EilSte}, \cite[Lemmas\,1,\,2]{Lin}, 
 and \cite[Theorem~3.1, p.261]{EilSte}.

For preliminaries on projective  and injective  limits,  
 see, e.g., \cite[Chap. VIII, \S2]{EilSte}.  Recall that the projective limit of a projective system $((X_i)_{i\in I}, (f_{ij})_{i\le j\in I})$ of nonempty compact Hausdorff spaces is a nonempty compact Hausdorff space in the limit topology. 
 Henceforth, the projective limit will be denoted by  $\varprojlim X_i$, with the projective system $((X_i), (f_{ij}))$ being understood. 
 As usual, the abbreviation \textrm{ANR} stands for absolute neighbourhood retract (for the corresponding definition, see, e.g., \cite[p.80]{Hu}). Throughout the article, $\cong$ denotes an isomorphism of objects of a category in question, and $\homeq$   a homotopy of maps or a homotopy equivalence of topological spaces. 

\begin{lemma}
\label{proposition_25_may_2022_1510}$\;$ 
Suppose $X=\varprojlim X_i$ and $\pi_i: X\rightarrow X_i$ are canonical continuous projections, where the $X_i$ are nonempty compact Hausdorff spaces,  and $Y$ is an \textrm{\em ANR}. 
\begin{itemize}
\item[(1)] Given a continuous map $f:X\rightarrow Y,$  there exists an index $i$ and a continuous map $f_{i}:X_{i}\rightarrow Y$ such that  $f_{i}\circ \pi_{i}\homeq f$. 
\item[(2)] Given a finite rank complex vector bundle $E$ over $X,$ there exists an index $i $ and a complex vector bundle $E_{i}$ over $X_{i}$ such that the pullback bundle 
$\pi_{i}^*(E_{i}) \cong E$.
 Moreover$,$ if $E$ is stably trivial$,$ then $E_i$ is stably trivial as well.
\item[(3)] For each $k\in\mZ_+$, the \v{C}ech cohomology group 
$H^k(X,\mZ)= \cup_i\, \pi_i^*(H^k(X_i,\mZ))$.\footnote{In other words, $H^k(X,\mZ)\cong\varinjlim H^k(X_i,\mZ)$,  the  injective limit of the related injective system of groups $(H^k(X_i,\mZ), f_{ij}^*)$.}
\end{itemize}
\end{lemma}

The following result expresses continuity of projective freeness and Hermiteness for algebras $C(X)$ on compact Hausdorff spaces $X$. 
\begin{proposition}\label{5_June_2022_9:00}
Suppose $X\homeq \varprojlim X_i$,  
where the $X_i$ are nonempty compact Hausdorff spaces such that 
either  all   algebras $C(X_i)$ are projective free$,$ or 
 all   algebras $C(X_i)$ are Hermite. 
Then $C(X)$ is projective free or Hermite$,$ respectively. 
\end{proposition}
\begin{proof}
Due to Corollary \ref{cor2.4}, we can assume $X=\varprojlim X_i$.
Then the required result follows from Propositions~\ref{prop2.1}, \ref{prop2.3} (applied to the algebras $C(X_i)$) and Lemma~\ref{proposition_25_may_2022_1510}. 
\end{proof}

To formulate our next result, we recall some definitions. 
 
 The {\em covering dimension} of a topological space $X$, denoted by $\dim X$, is the smallest integer $d$  such that every open covering of $X$  has an open refinement of  order at most $d+1$. If no such integer exists,  then $X$ is said to have {\em infinite covering dimension}. 

Given a finite open cover $\alpha=\{U_1,\dots, U_r\}$ of a compact topological space $X$,  its {\em nerve} is the abstract
simplicial complex $A_\alpha$, whose vertex set is $\alpha$ and such that the simplex 
$\sigma = [U_{i_0} , U_{i_1} , \dots , U_{i_n}] \subset A_\alpha$ if and only if $\cap_{j=0}^n U_{i_j}\neq \emptyset$.

We also recall that  a path-connected topological space $X$ is {\em $n$-simple} if for each $x \in  X$, 
the fundamental group $\pi_1(X, x_0)$ acts trivially on the $n^{\textrm{th}}$ homotopy group $\pi_n(X, x_0)$, see, e.g., \cite[Chap. IV, \S16]{Hu00} for details. For instance, every path-connected topological group is $n$-simple for all $n\in \mN$.

The following result, generalising \cite[Corollary~1.4]{BruSas} and 
\cite[Theorem~5]{Tol}, gives a sufficient condition for $C(X)$ to be projective free or Hermite.

 \begin{theorem}
\label{teo2.8}
If $X\homeq \varprojlim X_i$, 
where the $X_i\,(\ne\emptyset)$ are  finite-dimensional compact Hausdorff spaces 
such that $H^n(X_i,\mZ)=0$ for all $n\geq 5,$  then $C(X)$ is Hermite. In addition$,$ if  all spaces $X_i$ are connected and satisfy $H^2(X_i,\mZ)=H^4(X_i,\mZ)=0,$ then $C(X)$ is projective free. 
\end{theorem}
\begin{proof}
In the proof, we use the fact that the rank of a stably trivial complex vector bundle or a complex vector bundle over a connected space is well-defined.
 Due to Proposition \ref{5_June_2022_9:00},
 it suffices to prove the result for each space $X_i$. 

 In what follows, we use the terminology from \cite{Hu0}. 
 So $A_\alpha$ denotes the nerve of a finite open covering $\alpha=\{U_1,U_2,\dots,U_r\}$ of $X_i$
 and $ \phi_\alpha: X_i\rightarrow A _\alpha$ denotes a canonical map of $\alpha$, i.e., 
 for each point $x\in X_i$, $\phi_\alpha(x)$ is contained in the closure of the simplex  
 $[U_{i_0},U_{i_1},\dots,U_{i_n}]$ of $A_\alpha$, where $U_{i_0},U_{i_1},\dots,U_{i_n}$ 
 denote the members of $\alpha$ containing $x$.  
 
 Given a continuous map $f: X_i\rightarrow Y$, 
 a continuous map $\psi_\alpha : A_\alpha\rightarrow Y$ is called a {\em bridge} if 
 $\psi_\alpha\circ\phi_\alpha\homeq f$ 
 for each canonical map 
 $\phi_\alpha: X_i\rightarrow A_\alpha$ of the covering $ \alpha$.  
The family $(A_\alpha)_{\alpha}$, together with the natural simplicial projection maps 
$p_\alpha^\beta: A_\beta\rightarrow A_\alpha$,  where $\beta$ is a finite refinement of $\alpha$, forms a projective system with  limit
$X_i$. 

 Let $E_i$ be a  stably trivial  complex  vector bundle  over $X_i$. 
Applying  Proposition~\ref{proposition_25_may_2022_1510}\,(2),
we obtain  for some $\alpha$, a stably trivial bundle $ E_\alpha$ over $A_\alpha$  such that 
$\phi_\alpha^*(  E_\alpha)\cong  E_i$;
in particular its Chern classes $c_k( E_\alpha)=0\, (\in H^{2k}(A_\alpha,\mZ))$, $k\in\mN$. If  $E_\alpha$ is of complex rank one, then the latter implies that $E_i$ is trivial. 
Next, suppose $E_\alpha$ has complex rank $n\geq 2$.  Let $A_\alpha^4$ be the $4^{\textrm{th}}$ skeleton of 
$A_\alpha$, that is, the totality of the simplexes of $A_\alpha$ with dimensions not exceeding $4$. Consider the bundle $ E':=  E_\alpha|_{A_\alpha^4}$. Since  the rank $n$ of $E_i $ is $\ge 2$, 
$E'$ is isomorphic to the Whitney sum $E^{\prime\prime} \oplus\Omega^{n-2}$, where $\Omega^{n-2}:=A_\alpha^4\times\mC^{n-2},$ 
  and $E^{\prime \prime}$ is a rank $2$ complex vector bundle over $A_\alpha^4$  (see \cite[Part~II, Chap.~9, Theorem~1.2]{Hus}). Since the Chern classes of $E' $ are zeros, the latter implies that Chern classes of $E^{\prime\prime}$ are zeros as well. In turn, 
$A_\alpha^4$ is four-dimensional, and therefore the vanishing of Chern classes of $E^{\prime \prime}$ implies that $E^{\prime\prime}$ is trivial (it follows, e.g., from \cite[Problem 14-C, p.171]{MilSta}).  
Hence, the bundle $E_\alpha|_{A_\alpha^4}=E'=E^{\prime\prime} \oplus\Omega^{n-2}$ is   trivial as well. In particular, due to the construction of \cite[Part~I, Theorem~3.5.5]{Hus}, there is a continuous map into a complex Grassmanian, $h: A_\alpha\rightarrow {\textrm{\bf  Gr}}_n(\mC^{n+k})$, constant in a neighbourhood of $A_\alpha ^4$,  such that
$h^*(\gamma_{n,n+k})\cong E_\alpha$, 
where $\gamma_{n,n+k}$ is the tautological bundle over $\textrm{\bf Gr}_n(\mC^{n+k})$.  Let the value of $h$ on $A_\alpha^4$ be $x$.
Consider the continuous maps $f_0 :=h \circ\phi_\alpha: X_i\rightarrow {\textrm{\bf  Gr}}_n(\mC^{n+k})$ and 
$f_1: X_i\rightarrow {\textrm{\bf  Gr}}_n(\mC^{n+k})$ of constant value $x$. 
According to the definition of $n$-homotopy\footnote{Let $X$ be a normal space, $Y$ be a connected ANR, and either $X$ or $Y$ be compact. 
Continuous maps $f_0,f_1:X\rightarrow Y$ are {\em $n$-homotopic} if there exists a bridge $\alpha$ for the pair $(f_0,f_1)$ with bridge map $\psi_i:A_\alpha\rightarrow Y$ for $f_i$ ($i=0,1$) such that $\psi_0=\psi_1$ on the $n$-dimensional skeleton $A_\alpha^n$ of the  nerve $A_\alpha$ of the covering $\alpha$ of $X$.} in \cite[\S6]{Hu0}, $f_0$ and $f_1$ are $4$-homotopic. Moreover, by the hypotheses we have $H^s(X_i,\mZ)=0$ for all $s\ge 5$ which implies that $H^s(X,G)=0$ for all $s\ge 5$, for  any abelian group $G$ (by the Universal Coefficient Theorem for \v{C}ech cohomology \cite[Chap.~6, \S8, Theorem~10]{Spa}).
In particular, $H^s(X,\pi_s)=0$ for all $s\ge 5$ where $\pi_s$ is the $s^{\textrm{th}}$ homotopy group of 
${\textrm{\bf  Gr}}_n(\mC^{k+n})$.
Since the latter is a simply connected  manifold, it is $r$-simple for all $r\ge 1$. Hence, the previous facts based on \cite[Theorem~8.1]{Hu} imply that $f_0$ and $f_1$ are homotopic maps. Thus $f_0^*(\gamma_{n,n+k})$ and $f_1^*(\gamma_{n,n+k})$ are
isomorphic bundles over $X_i$. But the former is isomorphic to the original bundle $E_i$ by our construction, and the latter
is trivial. This proves that $E_i$ is  trivial. Thus every stably trivial complex vector bundle over $C(X_i)$ is trivial. Hence, all $C(X_i)$, and therefore $C(X)$, are Hermite due to Proposition~\ref{5_June_2022_9:00}.

Now, assume in addition that $X_i$ is connected and $H^2(X_i,\mZ)=H^4(X_i,\mZ)=0$. 
Let $E_i$ be a finite rank complex  vector bundle  over $X_i$. Then as above, there exist $\alpha$ and a complex vector bundle $ E_\alpha$ over $A_\alpha$,  such that 
$\phi_\alpha^*(  E_\alpha)\cong  E_i$. 
Due to Lemma~\ref{proposition_25_may_2022_1510}~(3), we can choose the $\alpha$ such that the Chern classes $c_i( E_\alpha)=0\, (\in H^{2i}(A_\alpha,\mZ))$, $i=1,2$. 
Hence, the complex vector bundle $E':=  E_\alpha|_{A_\alpha^4}$ is trivial. 
So arguing as above, we obtain that the bundle $E_i$ is trivial.
Thus every finite rank complex vector bundle over $C(X_i)$ is trivial. Hence, all $C(X_i)$, and therefore $C(X)$, are projective free by Proposition ~\ref{5_June_2022_9:00}.
\end{proof}

We deduce from the theorem the following result.

\begin{theorem}\label{teo2.9}
Let $X$ be a Hausdorff paracompact space of finite covering dimension such that $H^n(X,\mZ)=0$ for all $n\ge 5$. Then the algebra $C(X)$ is Hermite. If$,$ in addition$,$ $X$ is connected and $H^2(X,\mZ)=H^4(X,\mZ)=0,$ then $C(X)$ is projective free.
\end{theorem}
\begin{proof} 
Due to the Dieudonn\'e  theorem, a Hausdorff paracompact space is normal; see, e.g., \cite[Theorem~5.1.5]{Eng}.  So the embedding $X\hookrightarrow \beta X$ induces a natural isomorphism $H_f^*(X,\mZ)\cong H^*(\beta X,\mZ)$; 
here $H_f^*(X,\mZ)$ are \v{C}ech cohomology groups of $X$ 
defined with respect to finite open coverings, see \cite[p.\,282]{EilSte}. 
Further, for a paracompact space $X$, the groups $H^*(X,\mZ)$ coincide with the 
\v{C}ech cohomology groups of $X$ defined with respect to numerable open coverings. 
Thus, according to \cite[Corollary~(6.3)]{CadSie}, the   
group $H_f^n(X,\mZ)$ coincides with the \v{C}ech cohomology group 
$H^n(X,\mZ)$ for all $n\ge 2$. These facts and our hypotheses imply that 
$H^n(\beta X,\mZ)=0$ for all $n\ge 5$. Moreover, since $X$ is normal, 
$\dim \beta X=\dim X$, and  
in particular,  $\beta X$ is finite-dimensional; see, e.g., \cite[Theorem~7.1.17, p.390]{Eng}.
Thus, applying Theorem \ref{teo2.8} to $\beta X$, we obtain that the algebra $C(\beta X)$ is Hermite. Then Proposition \ref{prop2.3} implies that the algebra $C(X)$ is Hermite as well.

In the second case, the previous argument implies that  $H^2(\beta X,\mZ)=H^4(\beta X,\mZ)=0$ also. Moreover, if $X$ is connected, then $\beta X$ is connected as well. Then, under these additional assumptions, Theorem~\ref{teo2.8} implies that the algebra $C(\beta X)$ is projective free. Therefore due to Proposition~\ref{prop2.1}, the algebra $C(X)$ is projective free as well, as required.
\end{proof}

\begin{example}\label{ex2.10}
Let $X\subset\mR^5$ be a closed subset. Then  $H^{n}(X,\mZ)=0$ for all $n\ge 5$. (In fact, the 
\v{C}ech cohomology groups of a closed subset of $\mR^m$ are isomorphic to the injective limit of \v{C}ech cohomology groups of its open neighbourhoods; see, e.g., \cite[Chap.~6, \S1, Theorem~12, \S8, Corollary~8]{Spa}. Also, by a result due to Whitehead \cite[Theorem~3.2]{Whi},
an open subset $U$ of $\mR^m$ is homotopy equivalent to an $m-1$ dimensional simplicial complex $\Gamma\subset U$. As the dimension of $\Gamma$ is $m-1$, $H^n(U,\mZ)=H^n(\Gamma,\mZ)=0$ for all $n\ge m-1$.) Thus,  due to Theorem \ref{teo2.9}, the algebra $C(X)$  is Hermite. 
\end{example}

In connection with Theorems \ref{teo2.8}, \ref{teo2.9} the following question seems quite natural:

\begin{question}
Is there a Hausdorff topological space $X$ with $H^n(X,\mZ)\ne 0$ 
for some $n\ge 5$ such that the algebra $C(X)$ is Hermite?
\end{question}

\section{Stein Algebras}
\label{section_O(X)}

For basic facts about complex analytic spaces and Stein spaces we refer the readers to the book \cite{GraRem}. 

Let $\Gamma(X,\mathcal O_X)$ be the ring of global sections of the structure sheaf $\mathcal O_X$ on a finite-dimensional complex analytic space $(X, \mathcal O_X )$.  There is a natural algebra homomorphism $\,\hat{\,} : \Gamma(X,\mathcal O_X)\rightarrow C(X)$ with image $\mathcal O(X)$, the ring of holomorphic functions on $X$, injective if $(X, \mathcal O_X)$ is reduced.  A space $(X, \mathcal O_X )$ is said to be {\em Stein} if it is {\em holomorphically convex} (i.e., for each infinite discrete set $D\subset X$ there exists an $f\in\mathcal O(X)$ which is unbounded on $D$) and  {\em holomorphic separable} (i.e., for all $x,y\in X$, $x\ne y$, there exists an $f\in \mathcal O(X)$ such that $f(x)\ne f(y)$).

By the Cartan and Oka theorem, the nilradical $\mathfrak{n}(\mathcal O_X)$ of $\mathcal O_X$ (i.e., the union of nilradicals of stalks $\mathcal O_x$, $x\in X$)
 is a coherent sheaf of ideals on $X$ and so if $(X,\mathcal O_X)$ is Stein, then by Cartan's Theorem B  we have the following exact sequence of global sections of sheaves
 \begin{equation}
0\rightarrow \Gamma(X,\mathfrak{n}(\mathcal O_X))\rightarrow \Gamma(X,\mathcal O_X)\stackrel{r^*}{\rightarrow}\Gamma(X,\mathcal O_{{\rm red}\, X})\rightarrow 0,
\end{equation}
where 
$\mathcal O_{{\rm red}\, X}:=\mathcal O_X/\mathfrak{n}(\mathcal O_X)$ is the structure sheaf on the reduction of $X$. It easily seen that 
$\Gamma(X,\mathfrak{n}(\mathcal O_X))$ is the {\em Jacobson radical} of  $\Gamma(X,\mathcal O_X)$, i.e., the intersection of all maximal ideals of $\Gamma(X,\mathcal O_X)$; 
see, e.g., \cite[\S1.4]{For0}. Moreover, the algebra $\Gamma(X,\mathcal O_X)$ is $\Gamma(X,\mathfrak{n}(\mathcal O_X))$-{\em complete}, i.e., the natural homomorphism from $\Gamma(X,\mathcal O_X)$ to 
the projective limit of quotient algebras $\varprojlim \Gamma(X,\mathcal O_X)\big/\Gamma(X,\mathfrak{n}(\mathcal O_X))^N$ is an isomorphism\footnote{Equivalently, in the topology on $\Gamma(X,\mathcal O_X)$ determined by letting the family of ideals $\left\{\Gamma(X,\mathfrak{n}(\mathcal O_X)) ^N\right\}_{N\in\mN}$ be a base of open neighbourhoods of $0$, every Cauchy sequence converges to a unique limit.}, see, e.g., 
 \cite[Ch. V, \,\S4.3]{GraRem}.

\begin{theorem}\label{teo2.10}
Let  $(X,\mathcal O_X)$ be a finite-dimensional Stein space. The homomorphism $\,\hat{\,}:\Gamma(X,\mathcal O_X)\rightarrow C(X)$ induces a bijection between isomorphism classes of finitely generated projective $\Gamma(X,\mathcal O_X)$ and $C(X)$ modules.
\end{theorem}
\begin{proof} 
Since the algebra $\Gamma(X,\mathcal O_X)$ is $\Gamma(X,\mathfrak{n}(\mathcal O_X))$-{\em complete}, the correspondence
\begin{equation}
\label{29_6_2022_1822}
P\cong\Gamma(X,\mathcal O_X)\otimes_{\Gamma(X,\mathcal O_X)}\!P\stackrel{r^*\otimes {\rm id}_P}{\xrightarrow{\hspace*{07mm}}} \Gamma(X,\mathcal O_{{\rm red}\, X})\otimes_{\Gamma(X,\mathcal O_X)}\!P 
\end{equation}
determines a bijection between isomorphism classes of finitely generated projective $\Gamma(X,\mathcal O_X)$  modules and finitely generated projective $\Gamma(X,\mathcal O_{{\rm red}\, X})$ modules, see, e.g., \cite[Theorem 2.26]{Swa68}.

For the reduced Stein space $(X,\mathcal O_{{\rm red}\, X})$, the algebra $\Gamma(X,\mathcal O_{{\rm red}\, X})$ can be naturally identified with $\mathcal O(X)$.  
Then it follows from \cite[S\"{a}tze 6.7,\,6.8]{For1} (see also \cite[Theorem~2.1]{Mor})  that
there is a bijection between isomorphism classes of finitely generated projective $\mathcal O(X)$ modules and isomorphism classes of holomorphic vector bundles over $X$ of bounded rank.\footnote{I.e., the complex ranks of restrictions of such bundles to connected components of $X$ are uniformly bounded from above.} 
 Moreover, according to the Grauert theorem (see  \cite{Gra57}, \cite{Gra58} and \cite{Car}) the inclusion of sheaves $i:\mathcal O_{{\rm red}\, X}\hookrightarrow C_X$ (the sheaf of germs of continuous functions on $X$) induces a bijection between isomorphism classes of holomorphic and continuous complex vector bundles over $X$ of bounded rank. Next, since $X$ is a Hausdorff paracompact of finite covering dimension (by the definition of a finite-dimensional complex analytic space), each continuous complex vector bundle over $X$ is of bounded rank if and only it is of finite type (see, e.g., 
 \cite[Ch.\,3, Proposition\,5.4]{Hus}) 
 and, hence, by Swan's theorem (see \cite[Theorem~2]{Vas}), there is a bijection between isomorphism classes of continuous complex vector bundles over $X$ of bounded rank and of finitely generated projective $C(X)$ modules.  This implies that the correspondence
 \begin{equation}
\label{29_6_2022_1823}
P\cong \mathcal O(X)\otimes_{\mathcal O(X)}\!P\stackrel{i \otimes {\rm id}_P}{\xrightarrow{\hspace*{07mm}}} C(X)\otimes_{\mathcal O(X)}\!P 
\end{equation}
determines a bijection between isomorphism classes of finitely generated projective $\mathcal O(X)$ modules and finitely generated projective $C(X)$ modules.
 
 The composition of the  bijections in \eqref{29_6_2022_1822} and \eqref{29_6_2022_1823} gives the required statement: the correspondence
 \begin{equation}
P\cong\Gamma(X,\mathcal O_X)\otimes_{\Gamma(X,\mathcal O_X)}\!P\stackrel{\,\hat{\,}\,\otimes {\rm id}_P} {\xrightarrow{\hspace*{07mm}}} C(X)\otimes_{{\Gamma(X,\mathcal O_X)} }\!P 
\end{equation}
 determines a bijection between isomorphism classes of finitely generated projective $\Gamma(X,\mathcal O_X)$ modules and  finitely generated projective $C(X)$ modules.
\end{proof}

Theorems \ref{teo2.10} and \ref{teo2.9} imply the following:

\begin{theorem}\label{teo3.2}
The algebra $\Gamma(X,\mathcal O_X)$ is projective free or Hermite 
if and only if the algebra $C(X)$ is projective free or Hermite. \\
 In particular, if $H^n(X,\mZ)=0$ for all $n\geq 5$, then $\Gamma(X,\mathcal{O}_X)$ is Hermite, and if, in addition, $X$ is connected and $H^2(X,\mZ)=H^4(X,\mZ)=0$, then it is projective free.
\end{theorem}

\begin{example}$\;$
\label{ex3.3}

\noindent (1) According to \cite{Ham}, 
 a reduced Stein space $X$ of (complex) dimension $k$ is homotopy equivalent to a $k$-dimensional $CW$ complex. Hence, $H^n(X,\mZ)=0$ for all $n>k$. Thus, due to Theorem~\ref{teo3.2} if  $X$ is of dimension $\le 4$, then $\mathcal O(X)$ is Hermite, and if $X$ is one-dimensional and connected, then $\mathcal O(X)$ is projective free.
 
 \smallskip 
 
 \noindent (2) Let $U$  be an open subset of a Stein manifold $X$. Equipped with the topology of uniform convergence on compact subsets of $U$, the algebra  $\mathcal{O}(U)$ becomes a complex Fr\'{e}chet space.  Each nonzero homomorphism  $\mathcal{O}(U)\to\mC$ is an element of the dual space $\mathcal{O}(U)^*$ (see, e.g., \cite[Chap.~5, \S7.1]{GraRem}).
The space of such homomorphisms equipped with the weak-$*$ topology of $\mathcal{O}(U)^*$ is denoted by $M(\mathcal{O}(U))$.  If $f\in \mathcal{O}(U)$, then  $\hat{f}\in C(M(\mathcal{O}(U)))$ is defined by $\hat{f}(\alpha)=\alpha(f)$ for each $\alpha 
 \in M(\mathcal{O}(U))$.
 
Since $X$ is Stein, $M(\mathcal{O}(X))=X$ (see, e.g., \cite[Chap.~5, \S7]{GraRem}), so we have the natural restriction map
$\pi_U: M(\mathcal{O}(U))\rightarrow X$ given by $\pi_U(\alpha)(f)=\alpha(f|_U)$. Rossi \cite{Ros63} has shown that $M(\mathcal{O}(U))$ admits the structure of a Stein manifold in such a way that: (i) the map $U\rightarrow M(\mathcal{O}(U))$ sending $z\in U$ to the evaluation homomorphism at $z$ is a biholomorphism of $U$ with an open
subset of $M(\mathcal{O}(U))$ (we will regard $U$ as an open subset of $M(\mathcal{O}(U))$); (ii) if $f\in \mathcal{O}(U)$, then $\hat{f}$ is the unique holomorphic extension of $f$ to $M(\mathcal{O}(U))$ (so that $\mathcal{O}(M(\mathcal{O}(U)))\cong\mathcal{O}(U))$; (iii) $\pi_U$ is locally a biholomorphism. This and Theorem \ref{teo3.2} imply that \smallskip

\noindent $\bullet$ {\em if $X$ is of dimension $\le 4$, then the algebra $\mathcal O(U)$ is Hermite}.

\smallskip

Assume, in addition, that the set $U$ is holomorphically contractible (e.g., $X=\mC^k$ and $U\subset X$ is a star-shaped domain), then \smallskip

\noindent $\bullet$ {\em the algebra $\mathcal O(U)$ is projective free}. \smallskip

Indeed, let the holomorphic contraction be given by a continuous map $H: U\times [0,1]\to U$, so that  $H(\cdot,1)={\rm id}_U$, $H(\cdot, 0)=z_o\in U$, and $H(\cdot,t): U\to U$ is holomorphic for all $t\in [0,1]$. Then $H$ determines the map $H^*$ from $[0,1]$ to the set of homomorphisms $\mathcal O(U)\to \mathcal O(U)$ given by $(H^*(t))(f)=f(H(\cdot, t))$.  The transpose of each $H^*(t)$ induces a holomorphic map $\hat H(\cdot, t):M(\mathcal O(U))\to M(\mathcal O(U))$ such that $\hat H(\cdot, t)|_U=H(\cdot, t)$. Let us show that $\hat H: M(\mathcal O(U))\times [0,1]\to M(\mathcal O(U))$ is continuous. To this end, let $\{(z_n,t_n)\}_{n\in\mN}\subset M(\mathcal O(U))\times [0,1]$ be a sequence converging to $(z,t)\in M(\mathcal O(U))\times [0,1]$. For each $f\in\mathcal O(U)$, 
\[
{\scaleobj{0.97}{
\begin{array}{l}
\displaystyle
\varlimsup\limits_{n\to\infty}|(\hat H(z_n,t_n)-\hat H(z,t))(f)| \\[0.3cm]
\le\! \varlimsup\limits_{n\to\infty} |(\hat H(z_n,t_n)-\hat H(z_n,t))(f)|+ \varlimsup\limits_{n\to\infty} |(\hat H(z_n,t)-\hat H(z,t))(f)|\\[0.3cm]
=\!\varlimsup\limits_{n\to\infty} |z_n(f(H(\cdot, t_n))\!-\!f(H(\cdot, t)))|\!+\!\varlimsup\limits_{n\to\infty} |(z_n\!-\!z)(f(H(\cdot, t)))|\!=:\!I\!+\!II.
\end{array}}}
\]
By the definition of convergence in the weak-$^*$ topology, the limit $II$ equals $0$. Similarly, by continuity of $H$,  the sequence of
functions $\{f(H(\cdot, t_n))\}_{n\in\mN}\subset\mathcal O(U)$ converges uniformly on compact subsets of $U$ to the function $f(H(\cdot, t))\in\mathcal O(U)$. Since $z_n\in \mathcal O(U)^*$, the latter implies that the limit $\, I$ equals $0$ as well. Hence, $\lim_{n\to\infty}\hat H(z_n,t_n)=\hat H(w,t)$ in the topology of $M(\mathcal O(U))$, as required. 

Thus, $\hat H: M(\mathcal O(U))\times [0,1]\to M(\mathcal O(U))$ is a homotopy between $\hat H(\cdot,1)={\rm id}_{M(\mathcal O(U))}$ and 
$\hat H(\cdot,0)=z_o$, i.e., the Stein manifold $M(\mathcal O(U))$ is holomorphically contractible.
From here and Theorem \ref{teo3.2}, it  follows that the algebra $\mathcal O(U)\cong\mathcal O(M(\mathcal O(U)))$ is projective free.
 \end{example} 
                                          
\section{Commutative Unital Complex Banach Algebras}
\label{section_CUBA} 

Recall that for a commutative unital complex Banach algebra $A$,  the maximal ideal space $M(A)\subset A^\ast$  
 is the set of nonzero homomorphisms $A \!\rightarrow\! \mC\;$\,\footnote{Every such homomorphism is continuous, see, e.g., \cite[\S 23(A), Theorem]{Loo}.}\;\!endowed with the Gelfand topology, the weak-$\ast$ topology of  $A^\ast$. It is a compact Hausdorff space contained in the unit sphere of $A^\ast$. The Gelfand transform $\,\hat{\,} : A \rightarrow C(M (A))$, defined by $\hat{a}(\varphi):=\varphi(a)$ for $a\in A$ and $\varphi \in M(A)$, is a nonincreasing-norm morphism of Banach algebras.
  
  \begin{theorem}
 \label{14_may_2022_12:42}
 Let $A$ be a commutative unital complex Banach algebra. Then:
  \begin{itemize}
 \item[(1)] $A$ is projective free  if and only if  $ C(M (A))$ is projective free. 
\item[(2)] $A$ is Hermite if and only if $C(M(A))$ is Hermite. 
\end{itemize}
\noindent  In particular, if  
 $M(A)\homeq \varprojlim X_i$, 
where the $X_i$ are  finite-dimensional compact Hausdorff spaces  such that
$H^n(X_i,\mZ)=0$ for all $n\geq 5,$  then $A$ is Hermite. 
If, in addition, each space $X_i$ is connected and $H^2(X_i,\mZ)=H^4(X_i,\mZ)=0,$ then $A$ is projective free.
 \end{theorem}

Parts (1) and (2) of the theorem follow from a one-to-one correspondence (determined via the Gelfand transform) between  
the isomorphism classes of finitely generated projective  $A$ modules  and the isomorphism classes
 of complex vector bundles over $M(A)$ (see \cite{Nov}, and also \cite[\S 7.5, Theorem on p.199]{Tay})   along with the Swan theorem \cite[Theorem\,2]{Sw}.
In turn, the last statement  follows from Theorem~\ref{teo2.8}.

Let $\bar{A}$ be the uniform closure in $C(M(A))$ of the image under the Gelfand transform of algebra $A$. It is known that $M(\bar{A})=M(A)$, see, e.g., \cite[Proposition~3]{Roy}. Then we obtain from Theorem \ref{14_may_2022_12:42}: 

\begin{corollary}\label{cor4.2}
$A$ is projective free or Hermite if and only if  $ \bar{A}$ is  projective free or Hermite. 
\end{corollary}

\goodbreak

\begin{example}$\;$
\label{example_4.3}

\noindent 
(1)  Let $\textrm{L}^1[0, 1]$ be the Banach space of complex-valued Lebesgue integrable functions on $[0,1]$ with the norm $\|f \|_1:=\int_0^1 |f(t)| dt$. The space 
$\textrm{L}^1[0, 1]$ equipped with the multiplication given by truncated convolution $(f\ast g)(t):=\int_0^t f(\tau)g(t-\tau) d\tau$  becomes a complex commutative Banach algebra $V$ called the {\em Volterra algebra}. The algebra $V$ is non-unital without maximal ideals, see, e.g., \cite[Example~9.82]{MorRup}. Let $V_1$ denote the algebra of pairs $(f,c)$, where $f\in V$ and $c\in\mC$ with addition and multiplication given by
\[
(f,c)+(f',c'):=(f+f',c+c'),\quad (f,c)\cdot(f',c'):=(f\ast f'+c\cdot f'+c'\cdot f, c\cdot c').
\]
We equip $V_1$ with the norm $
\|(f,c)\|:=\|f\|_1+|c|$. Then $V_1$ becomes a commutative unital complex Banach algebra. Since $V$ is without maximal ideals, 
 $V\times \{0\}$  
is the only maximal ideal of $V_1$. Thus due to Theorem~\ref{14_may_2022_12:42} the algebra $V_1$ is projective free.

\smallskip 

\noindent (2) Let $A$ be a commutative unital complex Banach algebra such that the algebra $\bar A\subset C(M(A))$ is generated by $k$ elements (this is true, e.g., if $A$ itself is generated by $k$ elements). Then $A$ is Hermite if $k\le 5$, and projective free  if $k\le 2$ and $A$ does not contain nontrivial idempotent elements. Indeed,
in this case $M(A)$ is homeomorphic to a polynomially convex subset of $\mC^k$ (see, e.g., \cite[Chap. III, Theorem~1.4]{Gam}).  Recall that a compact set $K\subset \mC^k$ is {\em polynomially convex} if for every $\bbz\not\in K$ there is a polynomial $p\in \mC[z_1,\dots,z_k]$ such that $|p(\bbz)|> \sup_{\bbw \in K} |p(\bbw)|$. It is known, see, e.g., \cite[Corollary~2.3.6]{Sto},
that if $K\subset\mC^k$ is a compact polynomially convex set, then $H^n(K,\mZ)=0$ for all $n\ge k$. This and Theorem~\ref{14_may_2022_12:42} imply that $A$ is Hermite if $k\le 5$. If $k\le 2$ and $A$ does not contain nontrivial idempotent elements, then due to the Shilov idempotent theorem (see, e.g., \cite[Chap.~III, Corollary~6.5]{Gam}) $M(A)$ is connected. Hence, in this case Theorem~\ref{14_may_2022_12:42} implies that  $A$ is projective free. \smallskip

\noindent (3) Let $L^\infty(S)$ be the Banach algebra of  essentially bounded  measurable complex-valued functions on a measure space $S$,  with pointwise operations and the supremum norm. 
 Then the maximal ideal space $M(L^\infty(S))$ is totally disconnected (see, e.g., \cite[Chap.~I, Lemma~9.1]{Gam}) and, hence, ${\rm dim}\,M(L^\infty(S))=0$. Thus, Theorem~\ref{14_may_2022_12:42} implies that the algebra $L^\infty(S)$ is Hermite.

\smallskip 

\noindent (4) Let $\mD:=\{z\in \mC:|z|<1\}$ and $\mT:=\{z\in \mC:|z|=1\}$. With pointwise operations and the supremum norm, 
$\textrm{L}^\infty$ denotes the Banach algebra of essentially bounded Lebesgue measurable functions on $\mT$, and  $H^\infty$ the Banach algebra of all bounded holomorphic functions in $\mD$. Via identification with boundary values, $H^\infty$ is a uniformly closed subalgebra of $\textrm{L}^\infty$.  According to the Chang-Marshall theorem, see, e.g., \cite[Chap.~IX, \S 3]{Gar}, any uniformly 
closed subalgebra $A$ between $H^\infty$ and $\textrm{L}^\infty$ is a {\em Douglas algebra} generated by $H^\infty$ and a family $\mathscr B_A\subset\overline{H^\infty}$ of functions conjugate to some inner functions of $H^\infty$ (written $A=[H^\infty,\mathscr B_A]$). If $H^\infty\subsetneq A$, then the maximal ideal space $M(A)$ is a closed subset of 
$M(H^\infty)\setminus\mD$ of the form (see, e.g., \cite[Chap.~IX, Theorem~1.3]{Gar}):
\[
M(A)=\bigcap_{\bar u\in \mathscr B_A}\{x\in M(H^\infty)\, :\, |x(u)|=1\}.
\]
According to the results of Su\'{a}rez \cite{Sua}, for each closed set $K\subset M(H^\infty)$ 
 $\dim K \le 2$ 
and $H^2(K,\mZ)=0$. This and Theorem \ref{14_may_2022_12:42} imply that $A$ is Hermite and it is projective free if $M(A)$ is connected. Next, due to the Shilov idempotent theorem, $M(A)$ is connected if and only if $A$ does not contain nontrivial idempotents in $\textrm{L}^\infty$.   For instance, $M(A)$ is connected if $A$ is one of the algebras: $H^\infty$, $H^\infty+C$, where $C:=C(\mT)$,  or $B_1=[H^\infty,C_1]$ (the closed subalgebra generated by $H^\infty$ and  the Banach algebra $C_1$ of all complex-valued functions on $\mT$ which are continuous except possibly at $z=1$ but which have one-sided limits at $z=1$; for details, see e.g., \cite{Sar} or \cite[Chap.~IX, Exercise~7]{Gar}).  Thus,  
 in these special cases, $ A$ is projective free. 
\end{example}

\section{The class $\mathscr{C}$}
\label{section_class_C}

One way to construct new Banach algebras from known ones is to take their projective tensor product. In general, the projective tensor product of projective free or Hermite Banach 
algebras does not inherit the  property. In this  section we introduce a new class of projective free Banach algebras such that their projective tensor product with projective free or Hermite Banach algebras  
continues to be projective free or Hermite, respectively. 


 A topological space $X$ is said to be of {\em trivial shape} if every continuous map from $X$ to an ANR is homotopic to a constant map; see, e.g., \cite[p.248]{Mar99}.  A  space of trivial shape generalises the notion of a contractible space, and, in particular, if a  space of trivial shape is homotopy equivalent to an ANR, then it is contractible.
If $X$ is a  compact Hausdorff space of trivial shape, then it is connected and  \v{C}ech cohomology groups $H^k(X,\mZ)=0$ for all $k\ge 1$. 
   
 We say that a commutative unital complex Banach algebra $A$   belongs to the class $\mathscr{C}$ 
 if $M(A)$ is a  space of trivial shape. In this section, we study some properties of class $\mathscr{C}$.
 
 \smallskip

Let $B, C$ be unital closed subalgebras of
a commutative unital complex Banach algebra $\mathfrak A$, and let 
$B\widehat\otimes_{\mathfrak A}C\subset\mathfrak A$ be the closure of the subalgebra $\langle B,C\rangle$
 generated by $B$ and $C$. 
 Following \cite{All},  we assume that the following property is satisfied: 
\begin{quote}
There exists a constant $c$ such that for all $\xi\in  M(B)$ and every $n\in\mathbb N$ and  $b_i\in B$, $c_i\in C$, $1\le i\le n$,
\begin{equation}\label{eq3.1}
\left\|\sum_{k=1}^n\xi(b_k)c_k\right\|_{C}\le c\left\|\sum_{k=1}^n b_k c_k\right\|_{\mathfrak A}.
\end{equation}
\end{quote}

\begin{example}\label{remark_eq3.1}
(For basic definitions and results on topological tensor products, see, e.g.,  \cite{Rya}.)

Let $B, C$ be commutative unital complex Banach algebras and  let $B\widehat\otimes_{\alpha}C$ be the completion  of the algebraic tensor product $B\otimes C$ equipped with a {\em reasonable crossnorm} $\lVert\cdot\rVert_\alpha$, i.e., such that 
$\|v\|_\varepsilon\leq \|v\|_\alpha\leq \|v\|_\pi$ for all $v\in B\otimes C$.
Here $\lVert\cdot\rVert_\pi$ and $\lVert\cdot\rVert_\varepsilon$ denote the projective and injective tensor norms on $B\otimes C$, given for $v\in B\otimes C$ by  
$$
\!\!\!{\scaleobj{0.96}{
\begin{array}{rcl}
\|v\|_\pi \!\!\!\!\!&:=&\!\!\!\!\inf\Big\{\sum\limits_{i =1}^n \|b_i\|_B\, \|c_i\|_C\, :\, v=\sum\limits_{i=1}^n b_i\otimes c_i, \;\! n\in\mathbb N\Big\},\medskip\\
\|v\|_\varepsilon \!\!\!\!\!&:=&\!\!\!\!\sup\{|
(\xi\otimes\eta)(v)|\, :\,  \xi \in B^*\!, \
\|\xi\|_{B^*}\leq 1,\;  \eta \in C^*\!,\  \|\eta\|_{C^*}\!\leq \!1
\}.
\end{array}}}
$$
(As usual, $B^*$ and $C^*$ stand for duals of $B$ and $C$, respectively.)

Suppose that $B\widehat{\otimes}_\alpha C$ is a Banach algebra with operations compatible with operations on $B\otimes C$. Then  \eqref{eq3.1} is satisfied with $c=1$ when $\mathfrak{A}:=B\widehat{\otimes}_\alpha C$.
This is the case, e.g.,
 if $\mathfrak A$ is    (a) the projective  tensor product $B\widehat\otimes_{\pi}C$; (b)  the injective  tensor product $B\widehat\otimes_{\varepsilon}C$ where either $B$ or $C$  is  a uniform algebra; see, e.g., \cite[\S1.3]{DHT} for the references. (For other examples
 see, e.g., \cite[Theorem~4]{Sim}.)
 \end{example}

Let $i_B$ denote the embedding   $ B\hookrightarrow B \;\! \widehat{\otimes}_\mathfrak A C $.

\begin{theorem}\label{teo3.1}
If $C\in\mathscr C$, then  the correspondence 
\begin{equation}\label{eq3.2}
P\cong B\otimes_B\!P\stackrel{i_B\otimes {\rm id}_P}{\xrightarrow{\hspace*{07mm}}} (B \;\! \widehat{\otimes}_\mathfrak A C)\otimes_{B}\!P 
\end{equation}
determines a bijection between isomorphism classes of finitely generated projective $B$ modules and finitely generated projective $B \;\! \widehat{\otimes}_\mathfrak A C$ modules.
\end{theorem}

\begin{remark}\label{rem5.2}
The result shows that the class $\mathscr C$ is an analog of the class of local rings (i.e., those with unique maximal ideals) in commutative algebra (see, e.g., \cite{AM} for the corresponding definitions and results). Indeed, if $S$ is a local commutative ring with the maximal ideal $\mathfrak m$ and $R$ is a commutative ring, then for the completion in the Krull topology $\widehat{R\otimes S}_I:=\varprojlim(R\otimes S)/I^n$ with respect to the ideal $I=R\otimes \mathfrak m$, we obtain an analog of Theorem \ref{teo3.1}, i.e., the correspondence $R\mapsto \widehat{R\otimes S}_I$
induces a bijection between isomorphism classes of finitely generated projective $R$ modules and finitely generated projective $\widehat{R\otimes S}_I$ modules. Note that the canonical map $\pi$ from $R\otimes S$ to $\widehat{R\otimes S}_I$ is  injective on $R\otimes 1_S\, (\cong R)$.
If, in addition, $S$ is Noetherian, then by the Krull intersection theorem, $\cap_{n=1}^\infty\mathfrak m^n=0$ and, hence, $\pi$ is injective on $1_R\times S\, (\cong S)$ (here $1_S$ and $1_R$ are units in $S$ and $R$). Moreover, the subalgebra generated by
$\pi(R\otimes 1_S)$ and $\pi(1_R\otimes S)$ is dense in $\widehat{R\otimes S}_I $. Thus, in this case the completion $\widehat{R\otimes S}_I $ is an analog of $B\widehat\otimes_\mathfrak A C$ in Theorem \ref{teo3.1}.
\end{remark}

 To prove Theorem \ref{teo3.1}, first, we prove the following general result.
 
\begin{lemma}\label{lem3.2}
Under the condition \eqref{eq3.1}, 
$M(B\;\!\widehat\otimes_\mathfrak A C)$ is homeomorphic to $M(B)\!\times \!M(C)$.
\end{lemma}
\begin{proof}
Condition \eqref{eq3.1} implies that for $\xi\in M(B)$ 
the map $\langle B,C\rangle\rightarrow C$,
\begin{equation}\label{eq3.3}
\left( \sum_{k=1}^n b_kc_k  \right)\mapsto \sum_{k=1}^n\xi(b_k)c_k\in C,
\end{equation}
extends by continuity to a bounded multiplicative projection
\[
P_\xi:  B\widehat\otimes_{\mathfrak A}  C\rightarrow C.
\]
In particular,   $\eta\circ P_\xi\in M(B\;\!\widehat\otimes_\mathfrak A C)$ for each $\eta\in M(C)$. It is easily seen that the  map
\[
F: M(B)\!\times \!M(C)\rightarrow M(B\;\!\widehat\otimes_\mathfrak A C),\quad F(\xi,\eta):=\eta\circ P_\xi,
\]
is continuous. Also, $F$ is injective, as if $\eta_1\circ P_{\xi_1}=\eta_2\circ P_{\xi_2}$ for some $(\xi_i,\eta_i)\in M(B)\!\times \!M(C)$, $i=1,2$, then for all  $b \in B$, $c\in C$,
\begin{equation}\label{eq3.4}
\xi_1(b )\eta_1(c )=\xi_2(b )\eta_2(c ).
\end{equation}
Applying \eqref{eq3.4} with $c=1$ (the unit of $\mathfrak A$) and then again with $b=1$, we get $\xi_1=\xi_2$ and $\eta_1=\eta_2$, as required.

Further, the map $F$ is surjective, as if $\varphi\in M(B\;\!\widehat\otimes_\mathfrak A C)$, then clearly $\xi:=\varphi|_{B}\in M(B)$ and $\eta:=\varphi|_{C}\in M(C)$ and due to \eqref{eq3.3}
\[
\begin{split} 
F(\xi,\eta)\left( \sum_{k=1}^n b_kc_k  \right) & =(\eta\circ P_\xi)\left( \sum_{k=1}^n b_kc_k  \right)=\eta\left( \sum_{k=1}^n \varphi(b_k)c_k \right)\medskip\\
& = \sum_{k=1}^n \varphi(b_k)\varphi(c_k ) =\varphi\left( \sum_{k=1}^n b_kc_k  \right),
\end{split}
\]
i.e., $F(\xi,\eta)=\varphi$.

This completes the proof of the lemma.
\end{proof}

\begin{proof}[Proof of Theorem \ref{teo3.1}] 
Due to Lemma \ref{lem3.2} without loss of generality we will identify  
$M(B\;\!\widehat\otimes_\mathfrak A C)$ with $M(B)\!\times \!M(C)$. Then the transpose of $i_B$ restricted to  $M(B\;\!\widehat\otimes_\mathfrak A C)$ is the map $p_B:M(B)\!\times \!M(C)\rightarrow M(B)$,
$p_B(x,y)=x$ for all $(x,y)\in M(B)\!\times \!M(C)$.

According to the Novodvorski-Taylor theorem
(\cite{Nov},   \cite[\S 7.5]{Tay})
 and the Swan theorem \cite{Sw}, to prove the result we must show that the pullback by $p_B$ determines a bijection between isomorphism classes of complex vector bundles over $M(B)$ and $M(B)\times M(C)$. In turn, it suffices to prove the same for complex vector bundles of constant rank over clopen subsets $U\subset M(B)$ and  $U\times M(C)\subset M(B)\!\times \!M(C)$.

To this end, we present  $U$ as $\varprojlim U_i$, where all $U_i$ are finite-dimensional compact simplicial complexes
(cf. the argument of the proof of Theorem \ref{teo2.8}). Then 
$U\times M(C)=\varprojlim (U_i\times M(C)) $. Moreover, if  $\pi_i: U\rightarrow U_i$ are canonical   projections for the first limit, then $\tilde\pi_i:=(\pi_i, {\rm id}_{M(C)}): U\times M(C)\rightarrow U_i\times M(C)$ are canonical projections  for the  second one.

Suppose   $E$ is a complex vector bundle of rank $n$ over $U\times M(C)$.  Due to Lemma \ref{proposition_25_may_2022_1510}\,(2) there is an index $i $ and a complex vector bundle $E_{i}$ over $U_i\times M(C)$ such that the pullback bundle 
 $\tilde\pi_{i}^*(E_{i})\cong E$. 
 Then there is a  map $h\!\in \!C(U_i, {\textrm{\bf  Gr}}_n(\mC^{m}))\;$\footnote{For topological spaces $X,Y$ we denote by $C(X,Y)$ the set of continuous maps from $X$ to $Y$.} from $U_i$  into a complex Grassmanian such that 
$h^*(\gamma_{n,m})\cong E_i$, 
where $\gamma_{n,m}$ is the tautological bundle over $\textrm{\bf Gr}_n(\mC^{m})$, see, e.g., \cite[Part I, Theorem~3.5.5]{Hus}.  Consider the map $H: M(C)\rightarrow C(U_i,\textrm{\bf Gr}_n(\mC^{m}))$,
\[
H(y)(x):=h(x,y),\quad y\in M(C),\quad x\in U_i.
\]
Since  $U_i$ is an ANR (as it is a compact simplicial complex; see, e.g.,  \cite[Chap.~III, Corollary~8.4]{Hu}), and also $\textrm{\bf Gr}_n(\mC^{m})$ is an ANR (as it is a compact complex manifold \cite[Chap.~III, Corollary~8.3]{Hu}), the space $C(U_i,\textrm{\bf Gr}_n(\mC^{m}))$ of continuous maps from $U_i$ to $\textrm{\bf Gr}_n(\mC^{m})$ equipped with the topology of uniform convergence is an ANR as well; 
see, e.g. \cite[Chapter~VI, Theorem~2.4]{Hu}. Moreover, since every continuous map between compact metrisable spaces is uniformly continuous, the map $H$ is continuous. Thus by the definition of a space of trivial shape,  $H$ is homotopic to a constant map $M(C)\rightarrow C(U_i,\textrm{\bf Gr}_n(\mC^{m}))$. This homotopy gives rise to a homotopy between $h$ and a map $h_o: U_i\times M(C)\rightarrow  \textrm{\bf Gr}_n(\mC^{m})$, $h_o(x,y):=h(x,o)$ for all $(x,y)\in U_i\times M(C)$, where $o\in M(C)$ is a fixed point. In particular, we have the following isomorphisms of bundles
\begin{equation}\label{eq3.5}
E_i\cong h^*(\gamma_{n,m}) \cong h_o^*(\gamma_{n,m}) \cong p_i^*(E_i),
\end{equation} 
where $p_i: U_i\times M(C)\rightarrow U_i\times \{o\}$ is defined by the formula $p_i(x,y)=(x,o)$ for all $x\in U_i$.  Applying to \eqref{eq3.5} the pullback map $\tilde\pi_i^*$ we obtain that the bundle  $E$ is isomorphic to the pullback by the natural projection $U\times M(C)\rightarrow U\times \{o\}$ of the restriction of $E$ to  $U\times \{o\}$. Since $p_B$ maps $U_i\times \{o\}$ homeomorphically onto $U$, this shows that $E$ is isomorphic to a bundle pulled back by $p_B$ from $U$. Thus the map $p_B^*$ determines a surjection between isomorphism classes of complex vector bundles over $M(B)$ and $M(B)\times M(C)$. Clearly, it determines an injection between these sets as well, since if $E_1,E_2$ are bundles over $M(B)$ such that the pullback bundles $p_B^*(E_1)$ and $p_B^*(E_2)$ are isomorphic, then their restrictions to $M(B)\times \{o\}$ are isomorphic and so $E_1$ and $E_2$ must be isomorphic.

The proof of the theorem is complete.
\end{proof}

Theorem \ref{teo3.1} leads straightforwardly to the following:
  
 \begin{corollary}
 \label{theorem_24_may_2022_0929}
 Let $B,C\subset\mathfrak A$ be Banach algebras satisfying conditions of Theorem \ref{teo3.1}.
  If $C\in \mathscr{C}$ and $B$ is   Hermite, 
 then $B\;\!\widehat{\otimes}_\mathfrak A C$ is Hermite.
 If $C\in \mathscr{C}$ and $B$ is  projective free, 
 then $B\;\!\widehat{\otimes}_\mathfrak A C$ is projective free. 
 \end{corollary}
 
Let $B,C$ be commutative unital complex Banach algebras and
 $\|\cdot\|_\alpha$ be a norm on the  algebraic tensor product
$B\otimes C$. Let $\mathfrak A:=B\widehat{\otimes}_\alpha C$ be the completion of $B\otimes C$ with respect to $\|\cdot\|_\alpha$. 
Suppose that $\mathfrak A$ is a Banach algebra. 
Identifying $B$ and $C$ with subalgebras $B\otimes 1_C$ and $1_B\otimes C$ of $\mathfrak A$ (here $1_B$ and $1_C$ are units of $B$ and $C$), we assume that the triple $B,C,\mathfrak A$ satisfies condition \eqref{eq3.1}. Then we have:

\begin{proposition}\label{prop3.4}
If $B,C\in\mathscr C$, then $\mathfrak A\in\mathscr C$. 
\end{proposition}
\begin{proof}
According to Lemma \ref{lem3.2}, $M(\mathfrak A)$ is homeomorphic to the direct product $M(B)\times M(C)$. Let $f:M(B)\times M(C)\rightarrow X$ be a continuous map 
 to an ANR $X$.
Since by the hypotheses $M(C)$ is of trivial shape, repeating the arguments of the proof of Theorem \ref{teo3.1}, we obtain that $f$ is homotopic to the continuous map $f_o: M(B)\times M(C)\rightarrow X$, where $f_o(x,y):=f(x,o)$ for all $(x,y)\in M(B)\times M(C)$; here $o\in M(C)$ is a fixed point. In turn, since $M(B)$ is of trivial shape, the restriction $f_o|_{M(B)}$ is homotopic to a constant map. This implies that $f_o$ (and, hence, $f$) is homotopic to a constant map. Thus, $M(B)\times M(C)$ is of trivial shape, as required.
\end{proof}


\section{Examples}
\label{section_examples}
In Sections~6.1 and 6.2 we present some examples of algebras of class $\mathscr C$. We restrict ourselves to the case of semisimple algebras only. The choice of examples of Sections~6.3 and 6.4  is  based on  the research interests of the authors. 

\subsection{Banach algebras of symmetric functions}
\label{subsec_6.1}$\;$
Recall that the {\em polynomial convex hull} of a bounded set $K\subset \mC^n$  is the minimal polynomially convex set $\widehat K\Subset\mC^n$ containing $K$, i.e.,
$$
\widehat{K}:=\bigl\{\bbz \in \mC^n: |p(\bbz)|\leq \sup_{\bbw \in K} |p(\bbw)| \textrm{ for all } 
p\in \mC[z_1,\dots,z_n]\bigr\}.
$$

If $A$  is a finitely generated semisimple commutative unital complex Banach algebra with generators $f_1,\dots, f_n$, then the map 
\begin{equation}\label{eq4.4}
F(x):=(x(f_1),\dots,x(f_n)),\quad x\in M(A),
\end{equation}
is an embedding with image a polynomially convex subset of $\mC^n$ (see, e.g., \cite[Chap. III, Theorem~1.4]{Gam}).  If we identify $M(A)$ with $F(M(A))$, then  $A$ becomes a (not necessarily closed) subalgebra of $C(F(M(A)))$ such that  the restriction of polynomial algebra  $\mC[z_1,\dots,z_n]$ to $F(M(A))$ is dense in $A$. In what follows, we consider a more general situation of a unital complex Banach algebra $A\subset C_b(K)$ on a  subset $K\Subset\mC^n$ such that {\em $\mC[z_1,\dots,z_n]|_K$ is dense in $A$}. 
 We also assume that $A$ is {\em weakly inverse closed}, that is, it possesses the following property:
\begin{itemize}
\item[({\rm wi})]
\quad {\em If $f\in A$, and $\sup_{K}|f(x)|<1$, then $\frac{1}{1-f}\in A$.}
\end{itemize}
Under these conditions, $M(A)$ is naturally identified with $\widehat K$, see, e.g., \cite[Proposition\,1]{Roy} and \cite[Chap. III, Theorem~1.4]{Gam}. 

Let $G\subset GL_n(\mC)$ be a finite group of order $|G|$.
We say that the Banach algebra $A\subset C(K)$ is {\em $G$-invariant} if
$G(K)=K$ and $A$ is invariant with respect to the action of $G$ on $C(K)$:   $g^*(f)=f\circ g$ for  all $f\in C(K)$, $g\in G$. (Since $\mC[z_1,\dots, z_n]$ is invariant with respect to the action of $G$, this implies that $G(\widehat K)=\widehat K$ as well.) In this case, each $g^*: A\to A$, $g\in G$, is an automorphism of $A$, and since $A$ is semisimple, each $g^*$ is continuous (see, e.g., \cite[\S24B, Theorem]{Loo}).

For a $G$-invariant algebra $A$, the subalgebra $A_G\subset A$ of elements invariant with respect to the action of $G$ on $A$ (i.e., such $f\in A$ that
$g^*(f)=f$ for all $g\in G$) is said to be {\em $G$-symmetric}.

Let $\mathcal P_G\subset \mC[z_1,\dots, z_n]$ be the subalgebra of polynomials invariant with respect to the action of $G$. If $f\in A_G$ and  $(p_j)_{j\in\mN}\subset\mC[z_1,\dots, z_n]$ is a sequence such that $ f=\lim_{j\to\infty}p_{j}|_{K}$, then all
$ \tilde p_j:=\frac{1}{|G|}\sum_{g\in G} g^*(p_j)\in \mathcal P_G$ and $f=\lim_{j\to\infty}\tilde p_j|_K$ (because all $g^*:A\to A$ are continuous). 
Hence, $\mathcal P_G|_K$ is a dense subalgebra of $A_G$.  

By the Hilbert basis theorem, see, e.g., \cite[Theorem~7.5]{AM}, there exist  homogeneous  polynomials $h_1,\dots, h_m\in \mC[z_1,\dots,z_n]$, $m\ge n$, invariant with respect  to the action of $G$ 
 which generate $\mathcal P_G$.\footnote{For instance, if $G=S_n$ acts by permutation of coordinates, then we can choose $h_i$ to be the elementary symmetric polynomial of degree $i$, $1\le i\le n$.}
Hence, $A_G$ is generated by elements $h_1|_K,\dots,
 h_m|_K$. In particular, the continuous map $F_G: M(A_G)\to \mC^m$,
\begin{equation}\label{eq4.5}
F_G(x):=(x(h_1|_K),\dots,x(h_m|_K)),\quad x\in M(A_G),
\end{equation}
embeds $M(A_G)$ into $\mC^m$ as a polynomially convex subset. 
On the other hand, by Lemma~\ref{lem3.5} 
we obtain that $M(A_G)$ can be identified with the quotient space $\widehat K/G$. In this identification, if $\pi: \widehat K\to \widehat K/G$ is the quotient map, then
$A_G$ is isomorphic to a subalgebra $\tilde A_G\subset C(\widehat K/G)$ such that $ \pi^*(f)|_K\in A_G$ for all $f\in\tilde A_G$. Let $\tilde h_i\in \tilde A_G$ be such that $\pi^*(\tilde h_i)=h_i|_{\widehat K}$, $1\le i\le m$. Then the map $F_G$ becomes 
\[
F_G(x):=(
\tilde h_1(x)),\dots,\tilde h_m(x)),\quad x\in \widehat K/G.
\] 
This implies that the map $H:=(h_1,\dots, h_m):\mC^n\to\mC^m$ maps
$\widehat K$ onto $F_G(\widehat K/G)$.

\begin{remark}\label{rem4.1}
The image $H(\mC^n)$ is an $n$-dimensional complex algebraic subvariety of $\mC^m$. For each polynomially convex set $S\subset\mC^n$ invariant with respect to the action of $G$ on $\mC^n$, the uniform algebra $P(S)$ (defined as
the closure in $C(S)$ of  the algebra $\mC[z_1,\dots, z_n]|_{S}$) 
is $n$-generated, $G$-invariant, and $M(P(S))=S$. Hence,
the map $H$ determines a one-to-one correspondence $S\mapsto H(S)$ between $G$-invariant polynomially convex subsets of $\mC^n$ and
polynomially convex subsets of $\mC^m$ lying in $H(\mC^n)$. 
\end{remark}
\begin{proposition}
\label{prop4.2}
Suppose $A\subset C(K)$ is $G$-invariant for some finite  group $G\subset GL_n(\mC)$ and $K\Subset\mC^n$ star-shaped with respect to the origin. Then $A_G\in\mathscr C$.
In particular, $A_G$ is projective free.
\end{proposition}
\begin{proof}
Since  $K$ is star-shaped with respect to the origin and all $h_i$ are homogeneous polynomials, we have
$t\cdot\bbz\in K$ for all $\bbz\in K$, $t\in [0,1]$ and 
\[
H(t\cdot\bbz)=\bigl(t^{{\rm deg}\,h_1}\cdot h_1(\bbz),\dots, t^{{\rm deg}\,h_m}\cdot h_m(\bbz)\bigr).
\]
We set for $\bbw=(w_1,\dots, w_m)\in\mC^m$ and $t\in [0,1]$,
\[
D(\bbw,t):=\bigl(t^{{\rm deg}\,h_1}\cdot w_1,\dots, t^{{\rm deg}\,h_m}\cdot w_m\bigr).
\]
Then $D:\mC^m\times [0,1]\to\mC^m$ is continuous, maps $F_G(\widehat K/G)\times [0,1]$ into $F_G(\widehat K/G)$ and is such that $D(\cdot, 1)={\rm id}_{F_G(\widehat K/G)}$ and $D(\cdot, 0)={\bf 0}$. Hence, $F_G(\widehat K/G)$ is contractible and therefore $A_G\in\mathscr C$, as required.
\end{proof}

Let us consider several explicit examples of nonuniform  algebras $A$.

\smallskip 

\noindent (1) Let $\omega:  \mR_+\to\mR_+$ be a nondecreasing concave function, 
not identically zero, and  such that $\omega(0)=0$. We set 
\[
d(z_1,z_2):=\omega(\|z_1-z_2\|),\quad z_1,z_2\in\mC^n,
\]
where $\|\cdot\|$ is the Euclidean norm on $\mC^n\, (\cong\mR^{2n})$.

Then $d$ is a metric on $\mC^n$ compatible with its topology.

For a bounded set $K\subset \mC^n$, let $\textrm{Lip}_d(K) \subset C(K)$ be the Banach algebra of  complex-valued Lipschitz functions on $K$ with respect to $d$, equipped with norm
\[
\|f\|_{\textrm{Lip}}:=\sup_K |f|+\sup_{x\ne y}\frac{|f(x)-f(y)|}{d(x,y)}.
\]
Then $A$ is the completion in $\textrm{Lip}_d(K)$ of the algebra
$\mC[z_1,\dots, z_n]|_K$.  
Clearly, $A$ is weakly inverse closed, see  (wi).

\smallskip 

\noindent (2) Let $C^p(K)\subset C(K)$ be the restriction to $K$ of the algebra $C^p(\mC^n)$ of  bounded complex-valued functions $f$ on $\mC^n$ having bounded  continuous partial derivatives up to order $p$ with the norm the sum of supremum norms of $f$ and of all its partial derivatives. Equipped with the quotient norm, $C^p(K)$ becomes a unital commutative Banach algebra. Then 
$A$ is the completion in $C^p(K)$ of the algebra
$\mC[z_1,\dots, z_n]|_K$. Clearly, $A$  is weakly inverse closed, see  (wi).

\smallskip 

 If $G\subset GL_n(\mC)$ is a finite group and $K$ is invariant with respect to the action of $G$ on $\mC^n$, then
it is easy to check that algebras $\textrm{Lip}_d(K)$ and $C^p(K)$ are $G$-invariant. Since $\mC[z_1,\dots, z_n]|_K$ is invariant with respect to the corresponding action of $G$, this implies that the algebras $A_G$ of $G$-symmetric $\textrm{Lip}_d(K)$ and $C^p(K)$ functions are $G$-invariant as well. If, in addition $K$ is star-shaped with respect to the origin, then by  Proposition~\ref{prop4.2}, these algebras are projective free.

\subsection{Bohr-Wiener algebras} 
 \label{subsec_6.2} Let $G$ be a connected compact abelian group and let $\Gamma$ be its (multiplicative) character group. Thus $\Gamma$ consists of continuous homomorphisms of $G$ into the group $\mT$ of unimodular complex numbers and separates points of $G$. As $G$ is connected, $\Gamma$ can be made into a linearly ordered group (see e.g. \cite[8.1.8]{Rud}).   Let $\preccurlyeq$ be a fixed linear order such that $(\Gamma, \preccurlyeq)$ is an ordered group. We henceforth write $\Gamma$ additively and denote its identity element by $0$.  

Standard widely used examples of $\Gamma$ are $\mZ^k$ and $\mR^k$ ($k\in \mZ_+$) with a lexicographic ordering; here we use usual addition in $\mZ$ and in $\mR$, and  the discrete topology in both cases. 

For a nonempty set $I$, we denote by $\ell^1(I)$ the complex Banach space  of complex-valued sequences $\bba=(a_i)_{i\in I}$ with pointwise operations and the norm 
$$
\|\bba\|_1:=\sum\limits_{i\in I}|a_i| .
$$ 
If $\emptyset \neq J\subset I$, then we view $\ell^1(J)$ as a subset of $\ell^1(I)$. 

If $I=\Gamma$, then $\ell^1(\Gamma)$ is  
a commutative unital complex Banach algebra 
with  multiplication given by convolution: 
$$
(\bba \ast \bbb)_j=\sum\limits_{k\in \Gamma} a_k b_{j-k},\quad 
\bba=(a_j)_{j\in \Gamma},\;\bbb=(b_j)_{j\in \Gamma}\in \ell^1(\Gamma).
$$ 

The algebra $\ell^1(\Gamma)$ is semisimple and its maximal ideal space is $G$. The Gelfand transform
 $\,\hat{\,}\,: \ell^1(\Gamma)\to C(G)$ is given by the formula
\begin{equation}
\label{eq4.1}
\hat{\bba}(g):=\sum\limits_{j\in \Gamma}\, a_j \bbe_j(g), \quad g\in G,\quad \bba=(a_j)_{j\in \Gamma}\in \ell^1(\Gamma).
\end{equation}
Here $\bbe_j(g):=\langle j, g\rangle\in \mT$ is the action of $j\in \Gamma$ on  $g\in G$. 

The function $\hat{\bba}$ is called the {\em symbol} of  $\bba$ with  {\em Bohr-Fourier coefficients} $a_j$ and  with the {\em Bohr-Fourier spectrum} $\{j\in \Gamma: a_j\neq 0\}$. The image of $\ell^1(\Gamma)$ under $\,\hat{\,}\,$ is denoted by $W(G)$. 
For a subsemigroup $\Sigma$ of $\Gamma$, we denote by $W(G)_\Sigma$ 
the algebra of symbols of elements in $\ell^1(\Sigma)$. We let $C (G)_\Sigma$ be the closure of $W (G)_\Sigma$ in $C(G)$ (so that $C(G)_\Gamma =C(G)$). The notions of Bohr-Fourier  
 coefficients and spectrum are extended from functions in 
 $W(G)$ to $C(G)$ by continuity. The Bohr-Fourier spectrum of an element of $C(G)$ is at most countable, and  $C (G)_\Sigma$ coincides with the set of functions having the
 Bohr-Fourier spectra in $\Sigma$.
 
 Assume that $\Sigma\subset\Gamma$ is {\em pointed}, i.e., such that
 \begin{equation}\label{eq4.2}
 0\in \Sigma,\quad {\rm and}\quad \Sigma\cap (-\Sigma)=\{0\}.
 \end{equation}
(E.g.,  if $\Gamma_+=\{j \in \Gamma: 0\preccurlyeq j\}$, $\Gamma_{-}=\{j\in \Gamma: j \preccurlyeq 0\}$, then  $\Gamma_{+}$ and $\Gamma_{-}$ are pointed subsemigroups.)

 \begin{Theorem}[\mbox{\cite[Theorem~1.2]{BRS}}]
If $\Sigma$ is a pointed subsemigroup, then the algebras $W(G)_\Sigma$ and $C(G)_\Sigma$ belong to the class $\mathscr C$.
\end{Theorem} 

Next, for a commutative unital complex Banach algebra $B$ with norm $\|\cdot\|_B$, we let $\ell^1(\Sigma ,B)$ denote
the complex Banach space  of $B$-valued sequences $\bbb=(b_j)_{j\in \Sigma }$ with pointwise operations and the norm 
$$
\|\bbb\|_{1,B}:=\sum\limits_{j\in \Sigma }\|b_j\|_B .
$$ 
Then the projective tensor product $B\widehat{\otimes}_\pi W(G)_{\Sigma}$ consists of $B$-valued continuous functions on $G$ 
\[
\widetilde{\bbb}(g):=\sum\limits_{j\in \Sigma }\, b_j \bbe_j(g), \quad g\in G,\quad \bbb=(b_j)_{j\in \Sigma }\in \ell^1(\Sigma,B),
\]
with norm $\|\widetilde{\bbb}\|:=\|\bbb\|_{1,B}$.

\noindent In turn, the injective tensor product 
$B\widehat{\otimes}_\varepsilon C(G )_{\Sigma }$ is
the closure of the commutative algebra $B\widehat{\otimes}_\pi W(G )_{\Sigma }$ in $C(G ,B)$
equipped with norm $\|f\|:=\max_{g\in G }\|f(g)\|_B$.

Now Theorem \ref{teo3.1} implies the following:
 
\begin{Corollary}
If $\Sigma$ is a pointed subsemigroup, then the correspondences 
\[
\begin{split}
P\cong & B\otimes_B\!P\stackrel{({\rm id}_B\otimes 1)\otimes {\rm id}_P}{\xrightarrow{\hspace*{15mm}}} (B\widehat{\otimes}_\pi W(G )_{\Sigma })\otimes_{B}\!P \\
P\cong& B\otimes_B\!P\stackrel{({\rm id}_B\otimes 1)\otimes {\rm id}_P}{\xrightarrow{\hspace*{15mm}}} (B\widehat{\otimes}_\varepsilon C(G )_{\Sigma })\otimes_{B}\!P
\end{split}
\]
determine bijections between isomorphism classes of finitely generated projective $B$ modules and finitely generated projective $B\widehat{\otimes}_\pi W(G )_{\Sigma }$ and $B\widehat{\otimes}_\varepsilon C(G )_{\Sigma }$ modules, respectively.

In particular, if $B$ is projective free $($or Hermite$)$, then the algebras  $B\widehat{\otimes}_\pi W(G )_{\Sigma }$ and $B\widehat{\otimes}_\varepsilon C(G )_{\Sigma }$ are  projective free $($respectively, Hermite$)$. 
\end{Corollary}

 \subsection{Algebras of holomorphic semi-almost periodic functions}
 \label{subsec_6.3}  Recall that $f \in C_b(\mR)$ is {\em almost
periodic} if the family $\{S_\tau f :\tau \in \mR\}$ of its translates, where 
 $(S_\tau f) (x) := f (x + \tau)$ ($ x \in \mR$), is relatively compact in $C_b(\mR)$. 
Let $AP (\mR)$ be the Banach algebra of almost periodic functions endowed with the supremum norm. 

Let $\mT$ be the boundary of the unit disc $\mD$, with the counterclockwise orientation. For $s := e^{it}$, $t \in [0, 2\pi)$, let
$\gamma^k_s(\delta) := \{se^{ikx} : 0 \leq  x < \delta < 2\pi\}$, $k \in \{-1, 1\}$, 
be two open arcs having $s$ as the right and the left endpoints (with respect
to the orientation), respectively.

A function $f \in \textrm{L}^\infty (:= \textrm{L}^\infty(\mT))$ is {\em semi-almost periodic} if
for any $s \in  \mT$, and any $\varepsilon > 0$ there exist a number 
$\delta = \delta(s, \varepsilon) \in  (0, \pi)$ and
functions $f_k : \gamma^k_s(\delta) \rightarrow \mC$, $k \in \{-1, 1\}$, 
such that functions 
$$
\widetilde{f}_k(x) := f_k(se^{ik\delta e^x} ),\quad  
-\infty < x < 0, \quad k \in \{-1, 1\},
$$
 are restrictions of some almost periodic 
functions from $AP (\mR)$, and
$$
\sup\limits_{z\in \gamma^k_s (\delta)} |f (z) - f_k(z)| < \varepsilon, \quad k \in \{-1,1\}.
$$
By $SAP $ we denote the Banach algebra of semi-almost periodic
functions on $\mT$ endowed with the supremum norm. The algebra $SAP$ contains as a special case an algebra introduced by Sarason \cite{Sar77}
 in connection with some problems in the theory of Toeplitz operators.

It is easy to see that the set of
points of discontinuity of a function in $SAP $ is at most countable. For 
 a closed subset $S$ of $\mT$, we denote by $SAP (S)$ the Banach algebra 
 of semi-almost periodic functions on $\mT$ that are continuous on $\mT \setminus S$. 
 
 Next, for a semi-almost periodic function $ f \in SAP $, 
 $k \in \{-1, 1\}$, and a point $s \in S$, the {\em left}
($k = -1$) and the {\em right} ($k = 1$) {\em mean values of $f$ over $s$} are given by the
formulas
\begin{equation}\label{eq6.5}
M^k_s(f) := \lim\limits_{n\rightarrow \infty}\frac{1}{b_n-a_n} 
\int_{a_n}^{b_n} f(s e^{ik e^t}) dt, 
\end{equation}
where $(a_n)_{n\in \mN}$ and $(b_n)_{n\in \mN}$ are  arbitrary sequences of real numbers converging to $-\infty$ such that $\lim\limits_{n\rightarrow \infty}(b_n-a_n)=+\infty$. 

The {\em Bohr-Fourier coefficients of $f$ over $s$} can be then defined by the
formulas 
\begin{equation}\label{eq6.6}
a^k_\lambda (f,s):=M_s^k(f e^{-i\lambda \log^k_s}),
\end{equation}
where 
$$
\log^k_s(se^{ikx}):=\log x, \quad  0<x<2\pi, \quad k\in \{-1,1\}.
$$
The {\em spectrum of $f$ over $s$} is 
\begin{equation}\label{eq6.7}
\textrm{spec}^k_s(f):=\{\lambda\in \mR: a_\lambda^k(f,s)\neq 0\}.
\end{equation}

Let $\Sigma : S \times \{-1, 1\} \rightarrow 2^{\mR}$ be a set-valued map 
which associates with each $s \in S$, $ k \in \{-1, 1\}$,  a unital semi-group 
$\Sigma(s, k) \subset  \mR$. By $SAP_\Sigma(S) \subset
SAP (S)$,  we denote the Banach algebra of semi-almost periodic functions
$f$ with $\textrm{spec}^k_s (f ) \subset \Sigma(s, k)$ for all $s \in S$, $k \in \{-1, 1\}$.

 Let $H^\infty$ be the Banach algebra of bounded 
holomorphic functions in $\mD$ with the supremum norm. Then $SAP_\Sigma(S) \cap H^\infty$ is called the {\em algebra of holomorphic semi-almost period functions with spectrum} $\Sigma$.  
 
 If $f\in SAP_\Sigma(S) \cap H^\infty$, then
 \begin{equation}\label{eq6.8}
\textrm{spec}^{-1}_s(f)=\textrm{spec}^1_s(f)=:\textrm{spec}_s(f)
\end{equation}
 and, moreover,
\begin{equation}\label{eq6.9}
a^{-1}_\lambda (f,s)=e^{\lambda\pi}a^{1}_\lambda (f,s)\quad {\rm for\ each}\;\;\lambda\in \textrm{spec}_s(f),
\end{equation}
 see \cite[Proposition~3.3]{BruKin}. Thus, $SAP_\Sigma(S) \cap H^\infty=SAP_{\Sigma_{{\rm sym}}}(S) \cap H^\infty$, where $\Sigma_{{\rm sym}}(s):=\Sigma_{{\rm sym}}(s,\pm 1)=\Sigma(s,-1)\cap\Sigma(s,1)$ for all $s\in S$.

 Let $A_\Sigma^S$ be the closed subalgebra of $H^\infty$ generated by the disk-algebra
$A(\mD)$ and the functions of the form $ge^{\lambda h}$, where ${\rm Re}(h)|_{\mT}$ is the characteristic
function of the closed arc going in the counterclockwise direction from the
initial point at $s$ to the endpoint at $-s$ such that $s\in S$, $\frac{\lambda}{\pi}\in\Sigma_{{\rm sym}}(s)$ and 
$g(z):= z + s$, $z\in\mD$ (in particular, $ge^{\lambda h}$ has discontinuity at $s$ only). It is shown in \cite[Corollary~3.7]{BruKin} that 
\[
SAP_\Sigma(S) \cap H^\infty=A_\Sigma^S.
\]
Let $b^S_{\Sigma}$ denote the maximal ideal space of the algebra $
SAP_{\Sigma}(S) \cap H^\infty$. The structure of $b^S_{\Sigma}$ is described in \cite[Theorem~3.13]{BruKin}. Specifically, the transpose of the embedding $A(\mD)\hookrightarrow SAP_{\Sigma}(S) \cap H^\infty$ determines a continuous epimorphism  $a_{\Sigma}^S :\bar{\mD}\setminus S\to\bar{\mD}$ one-to-one over $\bar{\mD}\setminus S$ and such that for each $s\in S$ the fibre of $a_\Sigma^S$ over $s$ is homeomorphic to the maximal ideal space $b_{\Sigma_{\rm sym}(s)}(T)$ of the algebra $APH_{\Sigma_{\rm sym}(s)}(T)$ of  holomorphic almost periodic functions on the closed strip $T=\{z\in\mC\, :\, {\rm Im}(z)\in [0,\pi]\}$ with the spectrum in $\Sigma_{\rm sym}(s)$. 

Now, we have (cf. \cite[Theorem~1.2]{BRS}): 

\begin{Theorem}[\mbox{\cite[Theorem~3.19]{BruKin}}]
\begin{itemize}
\item[(1)] The map of  cohomology groups induced by embeddings $(a_{\Sigma}^S)^{-1}(s)\hookrightarrow b_\Sigma^S$, $s\in S$, produces an isomorphism 
\[
H^{k}(b_\Sigma^S,\mZ)\cong \bigoplus_{s\in S}H^{k}(b_{\Sigma_{\rm sym}(s)}(T),\mZ),\quad k\ge 1.
\]
\item[(2)]
Suppose that each $\Sigma_{\rm sym}(s)$ is a subset of $\mR_+$ or $\mR_{-}$. 
Then  $H^{k}(b_\Sigma^S,\mZ)=0$ for all $k\ge 1$ and the algebra $SAP_\Sigma(S) \cap H^\infty$ is projective free.
 \end{itemize}
 \end{Theorem}
 
 \subsection{Algebras of bounded holomorphic functions} 
\label{subsec_6.4} $\;$ 

\noindent Let $X$ be a connected Riemann surface such
that the Banach algebra $H^\infty(X)$ of bounded holomorphic functions on $X$ separates the points of $X$.
Since $X$ is homotopy equivalent to a one-dimensional $CW$-complex, Theorem \ref{teo2.8} and Proposition \ref{prop2.1} imply that
the algebras $C(X)$ and $C_b(X)$ are projective free.
The transpose of the isometric embedding  $H^\infty(X)\hookrightarrow C_b(X)$ induces a continuous map $p:\beta(X)\to M(H^\infty(X))$ of the maximal ideal spaces, with image ${\rm cl}\, X$,  the closure of $X$ (identified with the set of evaluations at points of $X$) in $M(H^\infty(X))$.  Since 
$\dim X=2$ 
and $H^2(X,\mZ)=0$,  $\dim \beta X=2$ and $H^2(\beta X,\mZ)=0$ as well (see the proof of Theorem \ref{teo2.9} for the references). Moreover, the map $p$ is identical on $X$. These make the following conjecture plausible (cf. the Question in \cite{Bru21}): 

\begin{conjecture}\label{con6.3}
The covering dimension of ${\rm cl}\, X$ is $2$ and the \v{C}ech cohomology group $H^2({\rm cl}\, X,\mZ)=0$.
\end{conjecture}

If the conjecture is true, it implies the validity of the following:\smallskip

\begin{conjecture}\label{con6.4}
If ${\rm cl}\, X\!=\!M(H^\infty(X))\;$\footnote{I.e., that the corona theorem is valid for $H^\infty(X)$.}, then $H^\infty(X)$ is a projective free algebra.
\end{conjecture}

\noindent Let us describe some classes of Riemann surfaces $X$ for which one of the conjectures is valid.

\smallskip

\noindent (1) Let $R$ be an unbranched covering of a bordered Riemann surface and $X$ be a domain in $R$ such that
inclusion $i:X\hookrightarrow R$ induces a monomorphism $i_*:\pi_1(X)\to\pi_1(R)$ of fundamental groups. The 
corona theorem, ${\rm cl}\, X=M(H^\infty(X))$, was proved in \cite[Corollary~1.6]{Bru04} and
projective freeness of $H^\infty(X)$ was established in \cite[Theorem~1.5]{BruSas}  using an analog of the classical Lax-Halmos theorem proved in \cite[Theorem~1.7]{Bru04b}. 
Thus Conjecture~\ref{con6.4} is valid  for such $X$.

In the special case $X=R$, it was proved in \cite[Theorem~1.3]{Bru21} that Conjecture~\ref{con6.3} is valid as well.
 
\begin{remark}\label{rem6.5}
Conjecture~\ref{con6.4} is false if $X$ has a nonempty corona, i.e., if $M(H^\infty(X))\setminus {\rm cl}\,X\ne\emptyset$. Indeed, using \cite[Theorem~1.2]{Bru03}, given an integer $n\ge 2$, one can construct an unbranched covering of a compact Riemann surface $X$ for which $H^\infty(X)$ separates points, 
 $\dim M(H^\infty(X))\ge n$, 
and $H^2(M(H^\infty(X)),\mZ) \ne 0$. The latter implies that $H^\infty(X)$ is not projective free.
\end{remark}

\noindent (2) A  wide class of planar domains, the so-called $\mathscr{B}$-domains, was introduced and
studied by Behrens \cite{Beh}. A set $X\subset \mC$ is called a {\em $\mathscr{B}$-domain} if  it is obtained from a domain $Y \subset \mC$ by deleting a (possibly finite) hyperbolically-rare sequence of closed discs $(\Delta_n)$ contained in $Y$ with centres $\alpha_n$, i.e., such that there are disjoint closed discs $D_n$ with centres $\alpha_n$ satisfying $\Delta_n\Subset D_n \subset Y$ and 
 $\sum \frac{\textrm{rad}\;\!\Delta_n}{\textrm{rad}\;\! D_n}<\infty$. 
It was shown in \cite[Theorem~1.1]{Bru21} that if $ X$ is obtained from a domain $Y \subset \mC$ by deleting a (possibly finite) hyperbolically-rare sequence of closed discs such that  (i) ${\rm cl}\, Y=M(H^\infty(Y ))$ and (ii) Conjecture~\ref{con6.3} is valid for $Y$, then ${\rm cl}\, X=M(H^\infty(X ))$ and Conjecture~\ref{con6.3} is valid for $X$ as well. (In particular, algebras $H^\infty(Y)$ and $H^\infty(X)$ are projective free.) 

The class of domains $Y$ satisfying (i), (ii) includes planar unbranched coverings of bordered Riemann surfaces (see part (1)) and domains obtained from them by deleting compact subsets of analytic capacity zero (e.g., totally disconnected compact subsets). 
Starting from such a $Y$ one can construct a descending chain $Y:= Y_0 \supset  Y_1 \supset Y_2 \supset \cdots  \supset Y_n$, $n \in \mN$,  of $\mathscr{B}$-domains, 
such that each $Y_i$ is defined by deleting a hyperbolically-rare
sequence of closed discs and then a compact subset of analytic capacity zero from $Y_{i-1}$.
Then all $Y_i$ satisfy assumptions (i), (ii), and, in particular, all algebras $H^\infty(Y_i)$ are projective free.

Recall that every bordered Riemann surface $S$ is a domain in a
compact Riemann surface $\widetilde S$ such that $\widetilde S\setminus S$ is the disjoint 
union of finitely many discs with
analytic boundaries (see \cite{Sto65}).
Here are some examples of planar unbranched coverings of $S$: 

\smallskip 

\noindent (a) It is known (see, e.g., \cite[Chap. X]{For29}) 
that each $\widetilde S$ is the quotient of a planar domain $\Omega$ by the discrete action of a
Schottky group $G$ (the free group with $g$ generators, where $g$ is the genus of $\widetilde S$) 
by M\"obius transformations. The corresponding quotient map $r : \Omega \rightarrow \widetilde S$ determines
the regular covering of $\widetilde S$ with the deck transformation group $G$. Then $Y := r^{-1}(S) \subset \Omega$ 
is a regular covering of $S$ satisfying conditions (i),(ii). By definition, $Y$ is the
complement in $\Omega$ of the finite disjoint union of $G$-orbits of compact simply connected
domains with analytic boundaries biholomorphic by $r$ to the connected components of
$\widetilde S \setminus S$. 

\smallskip 

\noindent (b) Consider the universal covering $r_u : \widetilde S_u \rightarrow \widetilde S$ of $\widetilde S$ 
(where $\widetilde S_u = \mD$ if $g \geq 2$, $\widetilde S_u = \mC$ if $g = 1$, and $\widetilde S_u = \mC\mP$ if $g = 0$), 
then $Y := r_u^{-1}  (S) \subset \widetilde S_u$ satisfies conditions (i),(ii)  as well. Here $Y$ is the complement in $\widetilde S_u$ of the finite disjoint union of
orbits under the action by M\"obius transformations of the fundamental group $\pi_1(\widetilde S)$ of $\widetilde S$ of
compact simply connected domains with analytic boundaries biholomorphic by $r_u$ to the
connected components of $\widetilde S \setminus  S$.
\smallskip

Finally, let us mention that if a connected Riemann surface $X$ is such that $H^\infty(X)$ separates its points, ${\rm cl}\, X=M(H^\infty(X))$,  
 $\dim {\rm cl}\, X=2$ 
and $H^2({\rm cl}\, X,\mZ)=0$, then by the K\"{u}nneth formula  (see, e.g., \cite[Chap.~6, Exercise~E]{Spa})  for $M(H^\infty(X))^i$, $2\le i\le 4$, the \v{C}ech cohomology groups $H^{n}(M(H^\infty(X))^i,\mZ)=0$ for all $n\ge 5$ and 
 $\dim M(H^\infty(X))^i =2i$.  
 Hence, due to Theorem~\ref{14_may_2022_12:42}, for all $2\le i\le 4$, the uniform algebra $(H^\infty(X))^{\widehat\otimes_\varepsilon i}$ (the $i^{\textrm{th}}$ injective tensor power of $H^\infty(X)$)  is Hermite.

\section{Appendix}

Let $A$ be a commutative unital complex Banach algebra and $G$ be a finite subgroup of automorphisms of $A$. By $A_G\subset  A$ we denote the Banach subalgebra of elements invariant with respect to the action of $G$, i.e.,
\[
A_G:=\{a\in A\, :\, g(a)=a\quad \forall g\in G\}.
\]
There is a natural projection  $P_G:A\to A_G$ given by the formula
\begin{equation}\label{eq3.6}
P_G(a):=\frac{1}{|G|}\sum_{g\in G} g(a) \;\; \textrm{ for all } a\in A.
\end{equation}
Here $|G|$ is the cardinality of the set $G$.

For each $g\in G$, the transpose of the map $a\mapsto g(a):A\rightarrow A$ induces a homeomorphism 
$g^*$ of the maximal ideal space $ M(A)$  of $A$. Thus we obtain an action of $G$ on $M(A)$. Let $M(A)/G$ be the quotient space by  this action  and $\pi: M(A) \rightarrow M(A)/G$ be the quotient map. We equip $M(A)/G$ with the smallest topology in which the map $\pi$ is continuous. Then $M(A)/G$ becomes a compact Hausdorff space homeomorphic to the maximal ideal space of the subalgebra $C(M(A))_G\subset C(M(A))$ of continuous functions invariant with respect to the action of $G$ on $M(A)$. We have:

\begin{lemma}\label{lem3.5}
$M(A)/G$ is homeomorphic to the maximal ideal space of the algebra $A_G$.
\end{lemma}
\begin{proof}
Let $\{g^*(x)\, :\, g\in G\}\subset M(A)$ be the equivalence class representing $\pi(x)\in M(A)/G$ for $x\in M(A)$. Since for each $a\in A_G$, 
\[
(g^*(x))(a)=x(g(a))=x(a) 
\]
(as elements of $A_G$ are invariant with respect to the action of $G$ on $A$), the functional $\varphi_z$, $z\in  M(A)/G$,
\[
\varphi_z(a):=x(a),\quad x\in \pi^{-1}(z),\quad \forall a\in A_G,
\]
is well-defined and is an element of $M(A_G)$. Thus, we have a  map $\Phi: M(A)/G\to M(A_G)$, $\Phi(z):=\varphi_z$.
Let us show that $\Phi$ is continuous. Indeed, by the definition of the Gelfand topology on $M(A_G)$, it suffices to show that for an open set $U_a:=\{y\in M(A_G)\, :\, |y(a)|< 1\}\subset M(A_G)$ with $a\in A_G$, the set $\Phi^{-1}(U_a)$ is open in $M(A)/G$. In turn, by the definition of topology on $M(A)/G$, the previous set is open if and only if its preimage under $\pi$ is open in $M(A)$.
We have 
\[
\begin{split}
\pi^{-1}(\Phi^{-1}(U_a))=&\;\! \{x\in M(A)\, :\, |\varphi_{\pi(x)}(a)|<1\}\\ = &\;\! \{x\in M(A)\, :\,|x(a)|<1\}.
\end{split}
\]
The latter is an open subset of $M(A)$ by the definition of the Gelfand topology.

Next, let us show that the map $\Phi$ is injective.

Let $z_1,z_2$ be distinct points of $M(A)/G$.  Since every finite subset of $M(A)$ is 
interpolating for  $A$,  there is $a\in A$ such that
\[
x(a)=0\;\; {\rm for\ all}\;\; x\in \pi^{-1}(z_1),\quad {\rm and}\quad
y(a)=1\;\; {\rm for\ all}\;\; y\in \pi^{-1}(z_2).
\]
Then, for $P_G(a)\in A_G$ (see \eqref{eq3.6}) we obtain 
\[
\varphi_{z_1}(P_G(a))=0,\quad {\rm and}\quad  \varphi_{z_2}(P_G(a))=1,
\]
i.e., $\Phi(z_1)\ne\Phi(z_2)$. Therefore $\Phi$ is injective and, hence, embeds $M(A)/G$ as a compact subset of
$M(A_G)$. 

Finally, let us show that the map $\Phi$ is onto.
If, on the contrary, there is  $y\in M(A_G)\setminus  \Phi(M(A)/G)$, then by the definition of the Gelfand topology on $M(A_G)$, there exist $a_1,\dots, a_k\in A_G$ such that 
\begin{equation}\label{eq3.9}
y(a_1)=\cdots =y(a_k)=0\quad{\rm and}\quad
\max_{1\le i\le k}\min_{z\in M(A)/G} |\varphi_z(a_i)|\ge\delta>0.
\end{equation}
This implies that 
\[
\max_{1\le i\le k} \min_{x\in M(A)}   |x(a_i)|\ge\delta>0.
\]
Equivalently, the family $a_1,\dots, a_n$ does not belong to a maximal ideal of $A$, i.e., there exist  $b_1,\dots, b_n\in A$ such that
\begin{equation}\label{eq3.10}
b_1a_1+\cdots +b_n a_n=1.
\end{equation}
We set
\[
\tilde b_i:=P_G(b_i),\quad 1\le i\le k.
\]
Then each $\widetilde b_i\in A_G$, and since $1$ and all $a_i\in A_G$, equation \eqref{eq3.10} implies that
\begin{equation}\label{eq3.11}
\tilde b_1a_1+\cdots +\tilde b_n a_n=1.
\end{equation} 
Applying $y$ to \eqref{eq3.11}, we get due to \eqref{eq3.9},
\[
1=y(1)=y(\tilde b_1)y(a_1)+\cdots +y(\tilde b_n)y( a_n)=0,
\]
a contradiction which shows that $\Phi$ is a surjection. Therefore
$\Phi: M(A)/G\to M(A_G)$ is a homeomorphism, as required.
\end{proof}

\end{document}